\documentclass[a4paper,tbtags,reqno,twoside]{article}
\usepackage{amsmath}
\usepackage{amsthm}
\usepackage{amssymb}

\theoremstyle{plain}
\newtheorem{thm}{Theorem}[section]  

\newtheorem{lem}[thm]{Lemma}

\theoremstyle{definition}

\newtheorem{ex}{Example}[section]

\theoremstyle{remark}
\newtheorem{rem}{Remark}[section]

\newcommand\DN{\newcommand}

\DN\lref[1]{Lemma~\ref{#1}}
\DN\tref[1]{Theorem~\ref{#1}}
\DN\pref[1]{Proposition~\ref{#1}}
\DN\sref[1]{Section~\ref{#1}}
\DN\dref[1]{Definition~\ref{#1}}
\DN\rref[1]{Remark~\ref{#1}} 
\DN\corref[1]{Corollary~\ref{#1}}
\DN\eref[1]{Example~\ref{#1}}
\numberwithin{equation}{section}
\newcounter{Const} 

\DN\Ct[1]{\refstepcounter{Const}c_{\theConst}\label{#1}} 
\numberwithin{Const}{section}
\DN\cref[1]{c_{\ref{#1}}}

\DN\SSSm{S_m}
\DN\PP{\mathrm{P}} 
\DN\EE{\mathrm{E}}

\DN\R{\mathbb{R}}\DN\N{\mathbb{N}}
\DN\Q{\mathbb{Q}}\DN{\Z}{\mathbb{Z}}

\DN\limi[1]{\lim_{#1\to\infty}} 	
\DN\limz[1]{\lim_{#1\to0}}
\DN\limsupi[1]{\limsup_{#1\to\infty}} 	
\DN\limsupz[1]{\limsup_{#1\to0}}
\DN\liminfi[1]{\liminf_{#1\to\infty}} 	
\DN\liminfz[1]{\liminf_{#1\to0}}
\DN\sumii[1]{\sum_{#1=1}^{\infty}}
\DN\sumi[1]{\sum_{#1=0}^{\infty}}

\DN\map[3]{#1\!:\!#2\!\to\!#3}
\DN\ot{\otimes} 
\DN\ts{\!\times\!}\DN\PD[2]{\frac{\partial#1}{\partial#2}}
\DN\Rd{\R ^d}

\title{Tagged particle processes and their non-explosion criteria } 
\author{{Hirofumi Osada } } 

\DN\hsym{h _{\mathrm{sym}}}

\DN\Akr{\mathfrak{A}^{k}_{r}}
\DN\Akrr{\mathfrak{A}^{k}_{r+1}}
\DN\AkrN{\mathfrak{A}^{k}_{r,N}}

\DN\Srk{S _{r}^{k}}
\DN\Srr{S _{r}^{k}}
\DN\Sri{S _{r(i)}^{k}}
\DN\xsss{(x,\mathfrak{s})}
\DN\ysss{(y,\mathfrak{s})}
\DN\Cz{C^{\infty}_0 (\R^d)} 
\DN\Czk{C^{\infty}_0 (\Rdk )} 
\DN\Czi{C([0,\infty); \R^d)} 

\DN\CD{\dione }
\DN\CY{C^{\infty} (\Rd )\ot\dY } 
\DN\CzY{C^{\infty}_0 (\Rd )\ot\dY } 

\DN\di{\mathcal{D}_{\circ}}
\DN\dione{\di ^{1}}
\DN\dik{\di ^{k}}

\DN\dnu{\mathcal{D}^{\nu }}
\DN\dnuk{\mathcal{D}^{\nuk }}
\DN\dnuik{\di ^{\nuk }}
\DN\dnurD{\mathcal{D}^{\nuk }_{r,\mathrm{D}}}
\DN\dnurDN{\mathcal{D}^{\nuk ,N}_{r,\mathrm{D}}}

\DN\dnui{\di ^{\nu }}
\DN\dnuirD{\dnuk _{\circ ,r,\mathrm{D}}}
\DN\dnuirkD{\mathcal{D}_{\circ ,r,\mathrm{D}}^{\nukr } }
\DN\dnuirkDN{\mathcal{D}_{\circ ,r,\mathrm{D}}^{\nukr ,N} }
\DN\dnuirDN{\mathcal{D}_{\circ ,r,\mathrm{D}}^{\nuk ,N} }

\DN\dmirD{\mathcal{D}_{\circ ,r,\mathrm{D}}^{\mu }}
\DN\dmirDN{\mathcal{D}_{\circ ,r,\mathrm{D}}^{\mu ,N}}
\DN\dmikrDN{\mathcal{D}_{\circ ,r,\mathrm{D}}^{\mukr ,N}}
\DN\dmi{\di ^{\mu }}
\DN\dmzi{\di ^{\mu _0}}

\DN\dm{\mathcal{D}^{\mu}}
\DN\dmrD{\mathcal{D}_{r,\mathrm{D}}^{\mu }}
\DN\dmrrD{\mathcal{D}_{r,\mathrm{D}}^{\murk }}
\DN\dmrDN{\mathcal{D}_{r,\mathrm{D}}^{\mu ,N}}
\DN\dmrrDN{\mathcal{D}_{r,\mathrm{D}}^{\murk ,N}}

\DN\dY{\mathcal{D}^{Y}}
\DN\dYi{\di ^{Y}}
\DN\dXY{\mathcal{D}^{XY}}
\DN\dXYi{\di ^{XY}}

\DN\Omegaik{\Omega _{i}^{a}}
\DN\Omegaione{\Omega _{i}^{1}}
\DN\Omegaiz{\Omega _{i}^{0}}
\DN\Omegaionestar{\Omega _{i*}^{1}}
\DN\Omegaizstar{\Omega _{i*}^{0}}

\DN\Omegak{\Omega _{\infty }^{a}}
\DN\Omegaone{\Omega _{\infty }^{1}}
\DN\Omegaz{\Omega _{\infty }^{0}}

\DN\Omegaiastar{\Omega _{i*}^{a}}
\DN\Omegaa{\Omega _{\infty }^{a}}

\DN\mur{\mu _{r}}
\DN\murN{\mu _{r}^{N}}
\DN\murk{\mu _{r}^{k}}
\DN\nur{\nu _{r}}
\DN\nurk{\nu _{r}^{k}}
\DN\murkN{\mu _{r}^{k,N}}
\DN\nurN{\nu _{r}^{k,N}}
\DN\muz{\mu _{0}}
\DN\dxmu{dx \ts \mu }

\DN\mux{\mu _{\mathbf{x} }}
\DN\muk{\mu ^{{k}}}
\DN\mukr{\mu ^{{k}}_r}
\DN\nuk{\nu ^{{k}}}
\DN\nukr{\nu ^{{k}}_r}
\DN\Lm{L^2(\mu )}
\DN\Lmr{L^2(\murk )}
\DN\Lmz{L^2(\muz )}
\DN\LXY{L^2(\0 )}

\DN\LXr{L^2(\R^d\!\otimes\!\mathfrak{S},\rho^2 dx\!\times\!\mu)}
\DN\Lnu{L^2(\nu )}
\DN\Lnuk{L^2(\nuk )}
\DN\Lnuz{L^2(\0 )}
\DN\RdT{\Rd \ts \mathfrak{S}}

\DN\zaY{\PP _{\vartheta _{x}(\mathfrak{s}-\delta _x)}^{Y}}
\DN\zXY{\PP ^{XY}_{(x,\vartheta _{x}(\mathfrak{s}-\delta _x))}}

\DN\szir{\bar{\sigma }^{0}_{i}}
\DN\soir{\bar{\sigma }^{1}_{i}}
\DN\sojr{\bar{\sigma }^{1}_{j}}
\DN\sair{\bar{\sigma }^{a}_{i}}
\DN\sziir{\bar{\sigma }^{0}_{i-1}}
\DN\soiir{\bar{\sigma }^{1}_{i-1}}
\DN\szr{\sigma ^{0}_{r}}
\DN\sor{\sigma ^{1}_{r}}
\DN\szrN{\sigma ^{0}_{r,N}}
\DN\sorN{\sigma ^{1}_{r,N}}
\DN\tauzrx{\tau ^{0}_{r,x}}
\DN\tauor{\tau ^{1}_{r,x}}

\DN\E{\mathcal{E}}
\DN\Enu{\E ^{\nu }}
\DN\Enuk{\E ^{\nuk }}
\DN\Enukr{\E ^{\nukr }}
\DN\Enur{\Enu _{r}}
\DN\Em{\E ^{\mu }}

\DN\DDD{\mathbb{D}}
\DN\D{\mathbf{D}}

\DN\PPY{\PP ^{Y}}
\DN\EY{\E ^{Y}}
\DN\Ya{(\EY ,\dY )}
\DN\Xa{(\EXY ,\dXY )}
\DN\TY{T_{Y,t}}

\DN\xx{\nabla }
\DN\Ext{\EE_{(x,\mathfrak{s}) }^{\nu}}
\DN\ExtXY{\EE _{(x,\mathfrak{s}) }^{XY}}
\DN\Pxt{\PP _{(x,\mathfrak{s}) }}
\DN\Pxtnu{\PP _{(x,\mathfrak{s}) }^{\nuk }}
\DN\Pm{\PP ^{\mu}}
\DN\Pmt{\Pm _{\mathfrak{s}}}

\DN\Done{\mathbb{D}^{1}}
\DN\EXY{\mathcal{E}^{XY}}
\DN\DXY{\mathbb{D}^{XY}}
\DN\DY{\mathbb{D}^{Y}}

\DN\3{(\E ^{\mu},\dmi )}
\DN\4{(\E ^{\mu},\mathcal{D}^{\mu})}
\DN\6{(\Enu , \dnu )}
\DN\7{(\E ^{\mu},\mathcal{D}^{\mu},\Lm )}
\DN\9{(\Enuk , \dnuk ,\Lnuk )}
\DN\0{dx\ts \mu _0}

\begin{document}  \maketitle 
\begin{flushright}
\footnote{{\bf Address}:  
Faculty of Mathematics, Kyushu University, 
\\
Fukuoka 812-8581, JAPAN \\
{\bf E-mail}: osada@math.kyushu-u.ac.jp \\ 
{\bf Phone and Fax}: 0081-92-642-4349 \\
{\bf MSC 2000 subj. class.  60J60, 60K35, 82B21, 82C22}:   }
\end{flushright}

\begin{abstract}
We give a derivation of tagged particle processes from unlabeled interacting Brownian motions. We give a criteria 
of the non-explosion property of tagged particle processes. 
We prove the quasi-regularity of Dirichlet forms describing the environment seen from the tagged particle, which were used in previous papers 
to prove the invariance principle of tagged particles of interacting Brownian motions.
\end{abstract}
\section{ Introduction }\label{s:1}
 
Interacting Brownian motions (IBMs) in infinite dimensions 
 are diffusions $ \mathbf{X}_t = (X_t^i)_{i\in \Z} $ consisting of infinitely many particles moving in $ \Rd $ with  the effect of the external force coming from a self potential $ \map{\Phi }{\Rd }{\R \cup \{ \infty \}}$ and that of the mutual interaction coming from an interacting potential $ \map{\Psi }{\Rd \ts \Rd }{\R \cup \{ \infty \}} $ such that $ \Psi (x,y) = \Psi (y,x)$. 

Intuitively, IBMs are described by the infinitely dimensional 
stochastic differential equation (SDE) of the form 
\begin{align} \label{:11} 
& dX_t^i = dB_t^i - \frac{1}{2} \nabla \Phi (X_t^i) dt 
-\frac{1}{2} 
\sum _{j \in \Z , j\not =  i } \nabla \Psi (X_t^i , X_t^j ) dt 
\quad (i \in \Z )
.\end{align}
The state space of the process 
$\mathbf{X}_t = (X_t^i)_{i\in \Z}$ is $(\Rd )^{\Z}$ by construction. 
Let $\mathfrak{X}$ be the configuration valued process given by 
\begin{align} \label{:12} 
& \mathfrak{X}_t = \sum _{i\in \Z } \delta _{X_t^i}
.\end{align}
Here $ \delta _a $ denotes the delta measure at $ a $ and 
a configuration is a Radon measure consisting of 
a sum of delta measures. 
We call $ \mathbf{X}$ the labeled dynamics and 
$ \mathfrak{X}$ the unlabeled dynamics.

The SDE \eqref{:11} was initiated by Lang \cite{La1}, \cite{La2}. 
He studied the case $ \Phi = 0 $, and 
$ \Psi (x,y) = \Psi (x-y) $, $ \Psi $ 
is of $ C^3_0(\Rd ) $, superstable and regular 
in the sense of Ruelle \cite{ruelle2}. 
With the last two assumptions, 
the corresponding unlabeled dynamics 
$ \mathfrak{X}$ has Gibbsian equilibrium states. 
See \cite{shiga}, \cite{Fr} and \cite{T2} 
for other works concerning on the SDE \eqref{:11}. 

In \cite{o.dfa} the unlabeled diffusion was constructed by the 
Dirichlet form approach. This method gives a general and simple proof of construction, and allows us to apply singular interaction potentials such as Lennard-Jones 6-12 potential, hard core potential and so on. 
See \cite{yoshida}, \cite{ark} \cite{tane.udf}, 
and \cite{yuu.05} for other works concerning 
on the Dirichlet form approach to IBMs.

In this paper we are interested in the property 
of each labeled particle of the unlabeled particle system given by the Dirichlet form. 
Such labeled particles are called tagged particles. By construction the unlabeled IBMs $ \mathfrak{X} $ 
are conservative since 
they have invariant probability measures and 
their state spaces are 
equipped with the vague topology. 
However, each labeled particle may explode 
under the Euclidean metric on $ \Rd $ in general. 
The first purpose of the paper is 
to give a criteria for the 
non-explosion  of the labeled particles 
(\tref{l:25}).

Let us next assume the total system is translation invariant in space. More precisely, we assume 
the stationary measure $ \mu $ 
and the energy form $ \Em $ 
of the Dirichlet space are translation invariant. 
Then the process $ \mathfrak{X}$ starting from $ \mu $ is 
translation invariant in space. 
The above assumption means, for Ruelle's class potentials \cite{ruelle2}, 
$ \Phi =0 $ and $ \Psi (x,y)= \Psi (x-y) $. 

This type of infinite-dimensional diffusions 
has been studied by the motivation 
from the statistical physics. 
One of the archetypical problem in this field is 
to investigate the large time property (the diffusive scaling limit, say) of tagged particles in the stationary system. This problem was solved for the simple exclusion process, 
which is a lattice analog of 
the hard core Brownian balls, 
by Kipnis-Varadhan \cite{KV}. 
For this they establish the celebrated 
Kipnis-Varadhan invariance principle. 

As for the tagged particle problem of IBMs, 
Guo \cite{guo}, Guo-Papanicolau \cite{gp} 
initiate the problem. 
Later De Masi {\em et al} \cite{De} study the problem for IBMs 
by using the Kipnis-Varadhan invariance principle. 
In \cite{o.inv2}, we convert the Kipnis-Varadhan invariance principle to the Dirichlet form theory. 
As a result, we weaken the assumption on the $ L^2 $-integrability of the mean forward velocity. 
This enables us to apply the invariance principle to hard core Brownian balls \cite{o.inv2} 
and \cite{o.p}. 

In \cite{o.inv2} we consider Dirichlet forms describing the tagged particle process and the environment process. 
These two Dirichlet forms are necessary to apply 
the Kipnis-Varadhan theory to IBMs. 
Although we gave the out line of the proof 
of the quasi-regularity of these Dirichlet forms and the relation between these two processes and the original unlabeled diffusion, the details were postponed. The second purpose of the paper is to give these details 
(Theorems \ref{l:24}, \ref{l:26} and \ref{l:27}).  

We establish the quasi-regularity of the Dirichlet forms of 
$ k $-labeled dynamics (\lref{l:23}) 
and prove the identity between $ k $-labeled dynamics  
and additive functionals of unlabeled dynamics (\tref{l:24}). 
The $ 0 $-labeled dynamics are simply the unlabeled dynamics; the $ k $-labeled dynamics are the processes of the form 
$ (X_t^{1},\ldots,X_t^{k},\sum_{j\not= 1,\ldots,k}\delta _{X_t^{j}}) $. 
The quasi-regularity of the Dirichlet form of 
the $ 0 $-labeled dynamics has been already proved in \cite{o.dfa}. 
Although \lref{l:23} is a straightforward generalization of it, 
we give a proof here for reader's convenience. 
On the other hand, the proof of \tref{l:24} is 
complicated because there is no simple transformation between Dirichlet spaces 
of the $ 0 $-labeled dynamics and the $ k $-labeled dynamics. 
\tref{l:24} plays an important role not only in the present paper but also in \cite{o.sde}. In \cite{o.sde} \tref{l:24} is used to solve 
the infinite-dimensional SDE \eqref{:11} describing IBMs.  

The organization of the paper is as follows:
In \sref{s:2} we give a set up and main results. 
In \sref{s:3} we introduce a transformation of Dirichlet spaces. 
In \sref{s:4} we prove the identity between 
unlabeled dynamics and the labeled dynamics (\tref{l:24}). 
In \sref{s:5} we prove the quasi-regularity of tagged particle processes and environment processes (\tref{l:26} and \tref{l:27}). 
In \sref{s:6} we study a non-explosion criteria and prove \tref{l:25}. 
In \sref{s:7} we prove the quasi-regularity of Dirichlet forms describing the $ k $-labeled and other unlabeled particles.

\section{ Set up and main results } 
\label{s:2}
Let $ {S}$ be a connected closed set in $ \Rd $ such that 
$ {S}= \overline{({S}^{\mathrm{int}})} $; that is, 
$ {S}$ coincides with the closure of the open kernel of $ {S}$. 
Let $ \mathfrak{S}$ be the set of the configurations on $ {S}$, that is, 
\begin{align}\label{:20}
\mathfrak{S} = \{ \mathfrak{s} = \sum _i \delta _{s_i}\, ;\, 
& \ \mathfrak{s}(K )< \infty \text{ for all compact sets } 
K \subset {S}\} 
.\end{align}
We endow $ \mathfrak{S}$ with the vague topology. 
Then $ \mathfrak{S}$ becomes a Polish space because $ {S}$ is a Polish space 
(see \cite{resnick}). 
Let $ \mu $ be a probability measure on $ (\mathfrak{S},\mathcal{B}({S}) )$. 

We say a non-negative permutation invariant function 
$ \rho ^n $ on $ {S}^{k}$ is the $ n $-correlation function of $ \mu $ if 
\begin{align}\label{:rm}&
\int_{A_1^{k_1}\ts \cdots \ts A_m^{k_m}} 
\rho ^n (x_1,\ldots,x_n) dx_1\cdots  dx_n 
= \int _{\mathfrak{S}} \prod _{i=1}^{m} 
\frac{\mathfrak{s}  (A_i) ! }
{(\mathfrak{s}  (A_i) - k_i )!} d\mu
 \end{align}
for any sequence of disjoint bounded measurable subsets $ A_1,\ldots,A_m \subset {S}$ and a sequence of natural numbers 
$ k_1,\ldots,k_m $ satisfying $ k_1+\cdots + k_m = n $. 

For a subset $ A \subset {S}$ we define the map 
$ \map{\pi _{A }}{\mathfrak{S}}{\mathfrak{S}} $ by 
$ \pi _{A } (\mathfrak{s} ) = 
\mathfrak{s}( A \cap \cdot )  $. 
We say a function $ \map{f}{\mathfrak{S}}{\R } $ 
is local if $ f $ is $ \sigma[\pi _{ A }]$-measurable 
for some compact set $ A \subset {S}$. 
We say $ f $ is smooth if $ \tilde{f} $ is smooth, 
where $ \tilde{f}((s_i)) $ is the permutation invariant function in $ (s_i) $ such that 
$ f (\mathfrak{s} ) = \tilde{f} ((s_i)) $ for $ \mathfrak{s} = \sum _i \delta _{s_i} $. 

Let $ \di $  be the set of all local, smooth functions on $ \mathfrak{S}$. 
For $ f,g \in \di $ we set 
$ \map{\DDD [f,g]}{\mathfrak{S}}{\R } $ by 
\begin{align} \label{:20z} & 
\DDD [f,g](\mathfrak{s}) = 
\frac{1}{2} 
\sum _{ i } 
(\nabla _{s_i} \widetilde{f},
\nabla _{s_i} \widetilde{g})_{\Rd }
.\end{align}
Here 
$ \nabla _{s_i} =
(\PD{}{s_{i1}},\ldots,\PD{}{s_{id}})$ and 
$ s_i = (s_{i1},\ldots,s_{id}) \in {S}$ and 
$ \mathfrak{s} = \sum _i \delta _{s_i} $. Moreover, 
$ (\ , \ )_{\Rd } $ is 
the standard inner product of $ \Rd $.  
For given $ f $ and $ g $ in $ \di $, 
it is easy to see that the right hand side 
depends only on $ \mathfrak{s}$. 
So $ \DDD [f,g]$ is well defined. %

Let $ \Lm = L^2(\mathfrak{S}, \mu )$. 
We consider the bilinear form $ (\Em , \dmi ) $ defined by  
\begin{align}\label{:20a} &
\Em (f,g) = \int _{\mathfrak{S}}  \DDD [f,g] d\mu ,\quad 
\\ \label{:20b} &
\dmi = \{ f \in \di \cap \Lm ; \ \Em (f,f)< \infty  \} 
.\end{align}
We now assume \\
\thetag{M.1.0} 
$ \3 $ is closable on $ \Lm $. 
\\
\thetag{M.2} 
The $ n $-correlation function $ \rho ^n $ 
of $ \mu $ is locally bounded for all $ n $. 

We collect some known results. 
\begin{lem}[\cite{o.dfa}] \label{l:20} 
Assume \thetag{M.1.0} and \thetag{M.2}. 
Let $ \4 $ be the closure of $ \3 $ on $ \Lm $. 
Then we have the following. \\
\thetag{1} $ \7 $ is 
a quasi-regular Dirichlet space. 
\\
\thetag{2} There exists a diffusion 
$\PP ^{\mu }=(\{\PP ^{\mu }_{\mathfrak{s}}\}_{\mathfrak{s}\in \mathfrak{S}},\mathfrak{X})$ 
associated with 
$ \7 $. \\
\thetag{3} The diffusion $\PP ^{\mu }$ is reversible 
with respect to $ \mu $. 
\end{lem}
\begin{proof}
\thetag{1} follows from \cite[Theorem 1]{o.dfa}. 
In \cite[Theorem 1]{o.dfa} we assume 
$ {S}= \Rd $; the generalization to the present case 
is easy. 
\thetag{2} follows from \thetag{1} and the general theory of Dirichlet forms \cite{mr}. \thetag{3} is clear because $ \PP ^{\mu }$ has 
an invariant probability measure $ \mu $ and 
the Dirichlet form $ \4 $ is $ \mu $-symmetric. 
\end{proof}

Let $ \mathrm{Cap}^{\mu } $ denote the capacity 
associated with the Dirichlet space $ \7 $. 
We refer to \cite[64 p.]{fot} for the definition 
of the capacity. 
We remark that the diffusion $ \PP ^{\mu }$ in \lref{l:20} \thetag{2} 
is unique up to quasi everywhere (q.e.). Namely, if 
$\hat{\PP }^{\mu }  = 
(\{\hat{\PP }^{\mu }_{\mathfrak{s}} \}_{\mathfrak{s}\in \mathfrak{S}},
\mathfrak{X})$ 
is another diffusion associated with 
$ \7 $, then there exists a set $ \hat{\mathfrak{S}}$ such that 
$ \mathrm{Cap}^{\mu } (\hat{\mathfrak{S}}^c) = 0 $ and that 
$ \PP ^{\mu }_{\mathfrak{s}}= \hat{\PP }^{\mu }_{\mathfrak{s}} $ 
for all $ \mathfrak{s}\in \hat{\mathfrak{S}} $.

We assume: 

\medskip
\noindent 
\thetag{M.3} 
$ \mathrm{Cap}^{\mu } 
(\mathfrak{S}_{\mathrm{single}}^c )=0  $. 

\medskip 
\noindent 
Here 
$ \mathfrak{S}_{\mathrm{single}}=
\{ \mathfrak{s}\in \mathfrak{S}\, ;\, 
\mathfrak{s}(x)\le 1 \text{ for all }x \in {S},\ 
\mathfrak{s}(x)\not= 0 \text{ for some }x \in {S}
\}
$. 

\medskip
\begin{lem} \label{l:22}
Assume \thetag{M.1.0}, \thetag{M.2}, and 
\thetag{M.3}. Then there exists 
a subset $ \hat{\mathfrak{S}}_{\mathrm{single}} $ 
such that 
\begin{align} \label{:22g}&
\hat{\mathfrak{S}}_{\mathrm{single}} \subset 
\mathfrak{S}_{\mathrm{single}}
,\\\label{:22a}&
\mathrm{Cap}^{\mu } 
(\hat{\mathfrak{S}}_{\mathrm{single}}^c )=0 
,\\ \label{:22d} &
\PP ^{\mu }_{\mathfrak{s}}(\mathfrak{X}_t\in 
\hat{\mathfrak{S}}_{\mathrm{single}} 
\text{ for all } t)= 1 
\quad \text{ for all } 
\mathfrak{s}\in \hat{\mathfrak{S}}_{\mathrm{single}} 
.\end{align}
\end{lem}
\begin{proof}
By \thetag{M.3} and the general theory of Dirichlet forms we have 
\begin{align*}&
\PP ^{\mu }_{\mathfrak{s}}(\mathfrak{X}_t\in 
\mathfrak{S}_{\mathrm{single}} 
\text{ for all } t)= 1 
\quad \text{ for q.e.\ } 
\mathfrak{s}\in \mathfrak{S}_{\mathrm{single}} 
.\end{align*}
Hence by taking a suitable version of 
$ \PP ^{\mu }$ we get a subset 
$ \hat{\mathfrak{S}}_{\mathrm{single}} $ 
satisfying \eqref{:22g}, \eqref{:22a}, and \eqref{:22d}.  
\end{proof}

We now introduce Dirichlet forms describing 
$ k $-labeled dynamics. 
For this we recall the definition of Palm measures. 
Let $  x =(x_1,\ldots,x_k)\in {S}^{k}$. 
We set 
\begin{align}\label{:23a}&
\mu _{ x  } = 
\mu (\cdot - \sum_{i=1}^{k} \delta _{x_i} | \ 
\mathfrak{s}( x_i )\ge 1 \text{ for }i=1,\ldots,k)
.\end{align}
Let $ \nuk $ be the measure defined by 
\begin{align}\label{:23b}& 
\nuk = \mu _{ x } \rho ^{k}( x ) 
d x  
.\end{align}
Here $ \map{\rho ^{k}}{{S}^{k}}{\R ^{+}} $ 
is the $ k $-correlation function of $ \mu $ 
as before,  and $ d x =dx_1\cdots dx_k  $ 
is the Lebesgue measure on $ {S}^{k} $. 
We set $ \nu = \nu ^{1} $ when $ k=1 $.  

Let $ \dik = C_0^{\infty}({S}^{k})\ot\di $. 
For $ f,g \in \dik $ 
let $ \nabla ^k[f,g] $ be such that 
\begin{align}\label{:23c}&
\nabla ^k[f ,g ] ( x ,\mathfrak{s})
= 
\frac{1}{2} \sum _{i=1}^{k} (
\nabla _{x_i} f ( x  , \mathfrak{s}),
\nabla _{x_i} g ( x  , \mathfrak{s}) )_{\Rd }
.\end{align}
where 
$ \nabla _{x_i} =
(\PD{}{x_{i1}},\ldots,\PD{}{x_{id}})$ and 
$  x =(x_1,\ldots,x_k)\in {S}^{k}$. 
We set $ \DDD ^{k} $ by 
\begin{align}\label{:23d}&   
\DDD ^{k}[f ,g ]( x  ,\mathfrak{s}) = 
\nabla ^k[f ,g ] ( x ,\mathfrak{s})
+
\DDD [f ( x ,\cdot ),g ( x ,\cdot )](\mathfrak{s})
.\end{align}

Let $ \Lnuk = L^2({S}^{k} \ts \mathfrak{S},\nuk ) $. 
We set $ (\Enuk ,\dnuik )$ by replacing $ \DDD $, 
$ \mu $ and $ \di $ in \eqref{:20a} 
and \eqref{:20b} with $ \DDD ^{k} $, $ \nuk $ and 
$ \dik $, respectively. 
For $ k\in \N $ we consider the assumption analogous to \thetag{M.1.0}. 

\medskip
\noindent 
\thetag{M.1.$ k $} \quad 
$ (\Enuk ,\dnuik )$ is closable on $ \Lnuk $. 

\medskip 

\begin{lem}\label{l:23}
Assume \thetag{M.1.0}, \thetag{M.1.$ k $}, and 
\thetag{M.2}. Let $ (\Enuk ,\dnuk )$ be 
the closure of $ (\Enuk ,\dnuik )$ on $ \Lnuk $. 
Then $ (\Enuk ,\dnuk )$ is 
a quasi-regular Dirichlet form on $ \Lnuk $. 
\end{lem}
By \lref{l:23} there exists a diffusion 
$ \PP ^{\nuk }= 
(\{\PP ^{\nuk }_{\xsss }\}_{\xsss \in {S}^{k}\ts \mathfrak{S}}, 
\mathfrak{X}^{1}) $ associated with 
the Dirichlet space $\9 $. 
Here we set $ \mathfrak{X}^{1}= (X,\mathfrak{X}) \in 
C([0,\infty);{S}^{k}\ts \mathfrak{S}) $. 
By construction 
$(x,\mathfrak{s})=\mathfrak{X}^{1}_0=(X_0,\mathfrak{X}_0)$ 
$ \PP ^{\nuk }_{\xsss }$-a.s.. 

Let $ \map{\kappa }{{S}^{k}\ts \mathfrak{S}}{\mathfrak{S}} $ 
be such that 
$ \kappa (x ,\mathfrak{s}) = 
\sum_{j=1}^{k}\delta _{x_j}+ \mathfrak{s}$, where 
$ x=(x_1,\ldots,x_k) $. 
By the correspondence 
$ \kappa ((X,\mathfrak{X}))=
\{\sum_{j=1}^{k}\delta _{X_t^{j}}+\mathfrak{X}_t\} $ 
we regard $ \kappa $ as the map from 
$ C([0,\infty);{S}^{k}\ts \mathfrak{S}) $ to 
$ C([0,\infty);\mathfrak{S}) $. 
We also denote by $ \kappa $ the map 
$ \map{\kappa }
{{S}^{\infty }\cup \sum_{k=1}^{\infty} {S}^{k}} 
{\mathfrak{S}}$ 
such that 
$ \kappa ((x_i)) = \sum _{i} \delta _{x_i}$, 
and regard $ \kappa $ as the map from 
$ C([0,\infty);
{S}^{\infty }\cup \sum_{k=1}^{\infty}{S}^{k}) $ to 
$ C([0,\infty);\mathfrak{S}) $.
For simplicity we denote these maps 
by the same symbol $ \kappa $.

Let $ \map{\jmath }{\mathfrak{S}_{\mathrm{single}}}
{{S}^{\infty }\cup \sum_{k=1}^{\infty}{S}^{k}} $ 
be a measurable map such that 
$ \kappa \circ \jmath $ is the identity map. 
We call this map a label map. 
Indeed, this map means labeling all the particles. 
We remark that plural maps satisfy 
the condition as above. 
So we choose any $ \jmath $ of such maps 
in the sequel. 

Once we fix a label map $ \jmath $, 
we can naturally extend the label map $ \jmath $ 
to the map from  
$ C([0,\infty);\mathfrak{S}_{\mathrm{single}}) $ to 
$ C([0,\infty);
{S}^{\infty }\cup \sum_{k=1}^{\infty}{S}^{k})$. 
Indeed, 
for a path $ \mathfrak{X}=\{ \mathfrak{X}_t  \}
\in C([0,\infty);\mathfrak{S}_{\mathrm{single}}) $, 
there exists a unique 
$ \{(X_t^i)\}\in C([0,\infty);
{S}^{\infty }\cup \sum_{k=1}^{\infty}{S}^{k})$ 
such that 
$ (X_0^i)=\jmath (\mathfrak{X}_0) $ and that 
$ \sum _i \delta _{X_t^i}=  \mathfrak{X}_t $ 
for all $ t \in [0,\infty ) $. 
We write this map by the same symbol $ \jmath $. 

\begin{thm} \label{l:24}
Assume \thetag{M.1.0}, \thetag{M.1.$ k $}, 
\thetag{M.2}, and \thetag{M.3}. Assume 
\begin{align}\label{:24c}&
\Pmt (\sup _{0\le t \le u}|X_t^{i}| < \infty 
\text{ for all } u, i \in \N )=1 
\quad \text{ for q.e.\ } \mathfrak{s}
.\end{align}
Here we initially label the process 
$ \mathfrak{X} $ as 
$ \mathfrak{X}_0 = 
\sum_{i=0}^{\infty}\delta _{X_0^i} $. 
Let $ \kappa $ and $ \jmath $ be maps given 
before \tref{l:24}. 
Let $ \hat{\mathfrak{S}}_{\mathrm{single}}  $ 
be as in \lref{l:22}. 
Then there exists a set $ \tilde{\mathfrak{S}} $ satisfying  
\begin{align}\label{:24e}&
\tilde{\mathfrak{S}} \subset 
\hat{\mathfrak{S}}_{\mathrm{single}}
,\\\label{:24a}&
\mathrm{Cap}^{\mu }(\tilde{\mathfrak{S}}^c)= 0 
,
\\\label{:24f}&
\PP ^{\mu }_{\mathfrak{s}}(\mathfrak{X}_t\in \tilde{\mathfrak{S}} 
\text{ for all } t)= 1 
\quad \text{ for all } \mathfrak{s}\in \tilde{\mathfrak{S}} 
,\end{align}
and for all $ k\in\N $ 
\begin{align}
\label{:24b} &
\PP _{\mathfrak{s}^{k}}^{\nuk }  = 
\Pm _{\kappa (\mathfrak{s}^{k})} \circ \jmath ^{-1}
\quad \text{ for all }\mathfrak{s}^{k}\in 
\jmath (\tilde{\mathfrak{S}})
\\ 
\label{:24}&
 \Pm _{\mathfrak{s}} = 
\PP _{\jmath (\mathfrak{s})}^{\nuk }\circ \kappa ^{-1} 
\quad  \text{ for all } \mathfrak{s}\in \tilde{\mathfrak{S}} 
.\end{align}
\end{thm}

\begin{rem}\label{r:24} 
\thetag{1} Since $ \jmath $ is any measurable map satisfying 
$ \kappa \circ \jmath = \mathrm{id.} $, 
we see by \eqref{:24} that 
$ \Pm _{\mathfrak{s}} = 
\PP _{\mathfrak{s}^{k}}^{\nuk }\circ \kappa ^{-1}  $ 
for all $ \mathfrak{s}^{k} \in \kappa ^{-1} (\mathfrak{s}) $. 
\\\thetag{2} 
Let $ \mathrm{Cap}^{\nuk } $ denote the capacity 
associated with $ \9 $. Then by \eqref{:24a} and \lref{l:41}, 
we deduce 
\begin{align}\label{:24d}&
\mathrm{Cap}^{\nuk }(\kappa ^{-1}(\tilde{\mathfrak{S}} )^c)= 0
.\end{align}
\end{rem}

We recall that $ \PP ^{\mu }$ is conservative 
as a diffusion on $ \mathfrak{S}$ equipped 
with the vague topology. 
However, each of the tagged particles may explode 
under the usual metric on $ \Rd $. 
So \eqref{:24c} does not hold in general. 
Next we prepare a sufficient condition for \eqref{:24c}. 

\begin{thm} \label{l:25}
Assume 
\thetag{M.1.0}, \thetag{M.1.1}, \thetag{M.2}, and 
\thetag{M.3}. 
Assume there exists $ T>0 $ such that 
for each $ R>0 $ 
\begin{align}\label{:25a}&
\liminf_{r\to \infty} \{ 
\int_{{S}_{r+R}} \rho ^1 (x)dx \}\cdot \ell 
(\frac{r}{\sqrt{(r+R )T}}) = 0
,\end{align}
where $ \ell (x)=(2\pi )^{-1/2}
\int_x^{\infty}e^{-x^2/2}dx $. 
Then we obtain \eqref{:24c}. 
\end{thm}

\begin{rem}\label{r:25}
\eqref{:25a} is satisfied if 
there exists a positive constant 
$ \Ct{;24} $ such that 
\begin{align}\label{:25b}&
\sup _{x \in {S}}\rho ^{1} (x) e^{-\cref{;24}|x|} < \infty  
.\end{align}
\end{rem}

We next proceed to the environment process. 
So we assume $ {S}= \Rd $. 
Let $ \map{\vartheta _{a}}{\mathfrak{S}}{\mathfrak{S}} $ denote 
the translation defined by 
$ \vartheta _{a}(\sum_i \delta_{x_i} )= 
\sum_i \delta_{x_i-a} $. 
We assume: 

\medskip
\noindent 
\thetag{M.4} $ \mu $ is translation invariant, 
that is, 
$ \mu \circ \vartheta _{a}^{-1} = \mu  $ 
for all $ a\in \Rd $. 

\medskip

\noindent 
By \thetag{M.4} we can and do choose 
the version $\mu _x $ in such a way that 
$\mu _x=\mu_0\circ\vartheta _x^{-1}$ 
for all $ x\in \R^d $. 
Here $ \mu _x $ is the conditional probability 
given by \eqref{:23a} with $ x \in \Rd $.

Let 
$\nabla _{i}=
(\PD{}{s_{i1}},\ldots,\PD{}{s_{id}})$. 
Let $\map{D}{\di }{(\di )^d}$ such that 
\begin{align}\label{:26a}&
D f(\mathfrak{s}) = 
\{ \sum_{i}\nabla _{i}\widetilde{f} \}  
\quad\text{ for }\mathfrak{s}\in \mathfrak{S}
\quad 
\text{($ \widetilde{f}$ is same as \eqref{:20z})}
.\end{align}
Note that $ D $ is the generator of 
the group of the unitary operators on $ \Lm $ 
generated by the translation $ \{\vartheta _a \}$.  
Let 
$ \xx = (\PD{}{x_{1}},\ldots,\PD{}{x_{d}})$ be the 
nabla on $ \Rd $. 
Let $\map{(D\!-\!\xx )}{\dione }{(\dione )^d}$ 
be such that 
\begin{align}\label{:26b}&
(D\!-\!\xx) f (x,\mathfrak{s}) = 
\{D f (x,\cdot )\} (\mathfrak{s}) - 
\{ \xx  f (\cdot ,\mathfrak{s}) \} (x)
\quad\text{ for }(x,\mathfrak{s})\in \mathfrak{S}^{1}
\end{align}
We set 
\begin{align}\label{:26c}&
\DY [f,g]=
\frac{1}{2}(D f ,D g )_{\Rd }
+\DDD [f,g]
\quad \text{ for $ f,g \in \di $}
\\ \notag &
\DXY [f,g]=
\frac{1}{2}
((D\!-\!\xx) f ,(D\!-\!\xx) g )_{\Rd }+ 
\DDD [f,g]
\quad \text{ for $ f,g \in \dione $}
.\end{align}
Here for $ f,g \in \dione $ we set 
$ \DDD [f,g](x,\mathfrak{s})= 
\DDD [f(x,\cdot ),g(x,\cdot )](\mathfrak{s}) $. 
Let 
\begin{align} \label{:26d}& 
\EY (f ,g )=\int_{\mathfrak{S}}\DY [f,g] d\mu_0 
\\ \label{:26e}& 
\EXY ( f , g )=\int_{\Rd \ts \mathfrak{S}}\DXY [f,g] dxd\mu_0
.\end{align}
Let $ \Lmz = L^2(\mathfrak{S},\mu_0)$ and 
$\LXY = L^2(\Rd \!\times\!\mathfrak{S},\0 )$. 
Let 
\begin{align}\label{:26f}&
\dYi = \{ g \in \di \cap \Lmz  \, ;\, 
\EY ( g , g )< \infty  \} 
\\\label{:26g}&
\dXYi = \{ h \in \CD \cap \LXY \, ;\, \EXY ( h , h )< \infty  \} 
.\end{align}

\begin{thm}   \label{l:26} 
Assume \thetag{M.1.0}, \thetag{M.1.1}, 
\thetag{M.2}, \thetag{M.3}, and \thetag{M.4}. 
Then \\
\thetag{1} 
The form $(\EY ,\dYi )$ is closable on $ \Lmz $. 
There exists a diffusion $ \PP ^{Y}$ 
associated with its closure $(\EY ,\dY )$ on $ \Lmz $. 
Moreover, 
$(\EY ,\dY )$ is a quasi-regular Dirichlet form on $ \Lmz $. 
\\\thetag{2} 
The form $(\EXY ,\dXYi )$ is closable on $ \LXY $. 
There exists a diffusion 
$\PP ^{XY}$ associated with its closure $\Xa $ on $ \LXY $. 
Moreover, 
$(\EXY ,\dXY )$ is a quasi-regular Dirichlet form on $ \LXY $. 
\end{thm}

\begin{rem}\label{r:22}
Fattler and Grothaus \cite{fg} prove the quasi-regularity of 
$(\EY ,\dY , \Lmz )$ and $(\EXY ,\dXY ,\LXY )$ 
for grand canonical Gibbs measures $ \mu $ with translation invariant interaction potentials which are differentiable outside the origin. 
Their method is different from ours. 
\end{rem}

By \eqref{:22d} we can write 
$\mathfrak{X}\in C([0,\infty);\mathfrak{S})$ as 
\begin{align}\label{:22b}&
\mathfrak{X}_t=\sum_{i} \delta _{X^i_t} 
\quad \PP ^{\mu }_{\mathfrak{s}}\text{-a.s.\ for all }\mathfrak{s}
\in 
\hat{\mathfrak{S}}_{\mathrm{single}}
,\end{align}
where $ X^i \in C(I_i ; \R^d)$ 
and $ I_i$ is the maximal interval in $[0,\infty)$ 
of the form $[0,b)$ or $(a,b)$ satisfying 
the representation \eqref{:22b}. 
Write $ \mathfrak{s}(x)=\mathfrak{s}(\{x\})$ and let 
\begin{align}\label{:22c}&
\mathfrak{S}_x = \{\mathfrak{s}\in \tilde{\mathfrak{S}} \ ;\ 
\mathfrak{s}(x)=1 \} 
.\end{align}
If $\mathfrak{X}_0=\mathfrak{s}\in \mathfrak{S}_x $, then 
there exists an $ i(x,\mathfrak{s})$ such that 
$ X_0^{i(x,\mathfrak{s})}=x $ and such 
$\R^d$-valued path 
$ X^{i(x,\mathfrak{s})}=\{ X_t^{i(x,\mathfrak{s})}\}$ is unique. 
For each $\mathfrak{s}\in \mathfrak{S}_x $ we regard 
$(X^{i(x,\mathfrak{s})},\PP ^{\mu }_{\mathfrak{s}})$ 
as the tagged particle starting at $ x $. 
Let $ \mathfrak{Y}^{x} $ be the process defined by 
\begin{align}\label{:26h}&
\mathfrak{Y}^{x}_t:=
\sum_{i\not= i(x,\mathfrak{s})} 
\delta _{X_t^i-X_t^{i(x,\mathfrak{s})}}
\quad\text{ under $\PP ^{\mu }_{\mathfrak{s}}$ } 
\text{ for  }\mathfrak{s}\in \mathfrak{S}_x 
.\end{align}
The process $ \mathfrak{Y}^{x} $ describes 
the environment seen from 
the tagged particle $ X ^{i(x,\mathfrak{s})} $.

Let $\PP ^{XY} $ be the diffusion associated with 
$(\EXY ,\dXY ,\LXY )$. 
The following clarifies the relations among 
the diffusions $\PP ^{\mu }$, $\PP ^{XY} $ and $ \PP ^{Y}$. 
\begin{thm} \label{l:27}
Assume \thetag{M.1.0}, \thetag{M.1.1}, 
\thetag{M.2}, \thetag{M.3}, and \thetag{M.4}. 
Let $ X^{i(x,\mathfrak{s})} $ and $ \mathfrak{Y}^{x} $ 
be as above. 
Let $\mathfrak{X}^{1}=(X ,\mathfrak{X} )\in {C([0,\infty); \RdT )} $. 
Then (a version of) 
$\PP ^{XY}$ satisfies for each $ x \in \Rd $ 
\begin{align}\label{:27a}&
\PP ^{\mu }_{\mathfrak{s}} ( X^{i(x,\mathfrak{s})} \in \cdot )
= 
\zXY 
( X \in \cdot ) \quad \text{ for all }
\mathfrak{s}\in \mathfrak{S}_x 
,\\ \label{:27b}&
\PP ^{\mu }_{\mathfrak{s}}(\mathfrak{Y}^{x} \in \cdot ) = 
\zXY 
(\mathfrak{X} \in \cdot ) = \zaY 
\quad \text{ for all }
\mathfrak{s}\in \mathfrak{S}_x 
.\end{align}
\end{thm}
\begin{rem}\label{r:12} 
The total system of interacting Brownian motions is 
a priori given by the diffusion $\PP ^{\mu }$. 
The diffusion $ \PP ^{Y}$ associated with $(\EY ,\dY , \Lmz )$ 
describes the motion of the environment seen from the tagged particle, and 
the diffusion $\PP ^{XY} $ associated with 
$(\EXY ,\dXY ,\LXY )$ corresponds to the motion of the coupling of the tagged particle and the environment seen from the tagged particle. \tref{l:26} and \tref{l:27} were used 
for the proof of the diffusive scaling limit of 
tagged particles of such translation invariant interacting Brownian motions in \cite{o.inv2}. 
\end{rem}

\begin{ex}\label{d:2}
\thetag{1} Let $ \mu $ be a canonical Gibbs measure 
with upper semicontinuous potentials. 
Assume the interaction potentials are 
super stable and regular in the sense of Ruelle. 
We refer to the reader \cite{ruelle2}. 
Then $ \mu $  satisfies 
\thetag{M.1.k} for all $ k $ and \thetag{M.2}. 
\thetag{M.3} is satisfies if $ d\ge 2 $ or the 
interaction potential has repulsive enough. 
See \cite{inu} for the necessary and sufficient condition for this when the number of particles are finite. 
Since the Dirichlet forms of the infinite particle systems are 
decreasing limits of the finite particle systems 
\cite{o.dfa}, Inukai's result gives a sharp sufficient condition of \thetag{M.3}. 
\\
\thetag{2} 
Let $ \mu $ be the Dyson's model in infinite dimension. 
This is a translation invariant probability measure on the one dimensional configuration space. Its correlation functions are given by the determinant of the sine kernel and related to the random matrices called GUE (see \cite{soshi.drpf}, \cite{mehta}). 
This measure satisfies 
\thetag{M.1}--\thetag{M.4}.  
Here \thetag{M.1} is the assumption that means 
\thetag{M.1.$ k $} hold for all $ k=0,1,\ldots $. 
We refer to \cite{o.rm} and \cite{o.col} for the proof of \thetag{M.1} and \thetag{M.2}, respectively. 
\\
\thetag{3} Let $ \mu $ be the Ginibre random point field. 
This is a translation invariant probability measure 
on the two  dimensional configuration space. 
Its correlation functions are given by the determinant of the exponential kernel 
and related to the random matrices called Ginibre Ensemble (see \cite{soshi.drpf}). 
This measure satisfies \thetag{M.1}--\thetag{M.4}. 
(see \cite{o.rm},\cite{o.col}). 
\\
\thetag{4} In \cite{o.rm} we introduce the notion of 
quasi Gibbs measures. This class contains all above examples. 
Measures in this class satisfies \thetag{M.1}. 
\end{ex}

\section{Transfer of Dirichlet spaces. } 
\label{s:3}
This section is devoted to the preparation 
of the proof of \tref{l:24}. 
We begin by considering 
the relation $ \mu $ and $ \nuk $ under 
the map $ \map{\kappa }{{S}^{k}\ts \mathfrak{S}}{\mathfrak{S}}$ 
defined before \tref{l:24}. 
Since these measures $ \mu $ and $ \nuk $ 
are not directly related by 
the map $ \map{\kappa }{{S}^{k}\ts \mathfrak{S}}{\mathfrak{S}}$, 
we consider the finite volume cut off of 
these measures instead.

Let $ S _{r}=\{ x \in {S}\, ;\, |x|<r \}  $ and 
$ \mathfrak{S}_{r,m}=\{ \mathfrak{s}\in \mathfrak{S}\, ;\, 
\mathfrak{s}(S _{r})= m \}$. We define the meaures 
$ \nukr $, $ \nurN $, $ \murk $, and 
$ \murkN $ by 
\begin{align}\label{:30}&
\nukr = \int _{\cdot }1_{S _{r}}(x) d\nuk ,\quad 
&&\nurN = \int _{\cdot }1_{\Srk }(x)
\sum _{m=1}^{N-1} 1_{\mathfrak{S}_{r,m}}(\mathfrak{s}) d\nuk 
\\\label{:!1}&
\murk = \nukr \circ \kappa ^{-1} ,\quad 
&&\murkN = \nurN \circ \kappa ^{-1}
.\end{align}
Let $ m^{[k]}=m(m-1)\cdots(m-k+1) $. 
Then it is not difficult to see that 
\begin{align}\label{:30a} &
\murk = \sum _{m=k}^{\infty} m^{[k]} \ 
\mu (\cdot \cap \mathfrak{S}_{r,m})
,\quad 
\murkN = 
\sum _{m=k}^{N} m^{[k]} \ 
\mu (\cdot \cap \mathfrak{S}_{r,m})
.\end{align}

Let $ \partial S _{r}=\{ |x|=r \}$ and 
$ \partial \mathfrak{S}_{r}= \{ \mathfrak{s}\in \mathfrak{S}; \ 
\mathfrak{s}(\partial S _{r})\ge 1  \}$. 
We remark $ \mathfrak{S}_{r,m}$ are open sets 
and their boundaries $ \partial \mathfrak{S}_{r,m} $ 
are contained in $ \partial \mathfrak{S}_{r}$. 
We define  $ \di ^{\mukr }$ 
in a similar fashion to $ \dnuik $ 
by replacing $ \nuk $ by $ \mukr $. 
Let 
\begin{align}\label{:31a}&
\dmirD = \{ f \in \dmi \, ;\, f(\mathfrak{s})=0 
\text{ if } \mathfrak{s}\in \partial \mathfrak{S}_{r}\}
,\\ \notag &
\dmirDN = 
\{ f \in \dmirD \ ; 
f (\mathfrak{s})=0 \text{ if } 
\mathfrak{s}\not\in  \sum _{m=1}^{N}\mathfrak{S}_{r,m}\ \}
,\\ \notag &
\dmikrDN = 
\{ f \in \di ^{\mukr } \ ; 
f (\mathfrak{s})=0 \text{ if } 
\mathfrak{s}\in \partial \mathfrak{S}_{r}\text{ or }
\mathfrak{s}\not\in  \sum _{m=1}^{N}\mathfrak{S}_{r,m}\ \}
.\\ \label{:31b}&
\dnuirD = 
\{ h \in \dnuik , ;\, 
h (x,\mathfrak{s})= 0\ \text{ if } 
x \not\in \Srk \text{ or } 
\mathfrak{s}\in \partial \mathfrak{S}_{r}\ \} 
,\\ \notag &
\dnuirDN = \{ h \in \dnuirD \, ;\, 
h (x,\mathfrak{s})= 0\ \text{ if } 
\mathfrak{s}\not\in  \sum _{m=0}^{N-1}\mathfrak{S}_{r,m}  \ \} 
.\end{align}
Let $ (\Em ,\dmrD )$ denote the closure of 
$ (\Em ,\dmirD )$ on $ \Lm $. 
We define the closures 
$ (\Em ,\dmrDN )$, 
$ (\E ^{\murk } ,\dmrrDN )$, 
$ (\Enuk ,\dnurD )$, and 
$ (\Enuk ,\dnurDN )$ similarly. 

For an $ h \in \dnuirD $ we set 
$ \hsym \in \dnuirD $ by 
\begin{align}\label{:31c}&
\hsym (x ,\mathfrak{s})= \frac{1}{m!} 
\sum _{\sigma \in \SSSm }
h (x_{\sigma (1)},\ldots,x_{\sigma (k)},
\sum_{i=k+1}^{m} \delta _{ x_{\sigma (i)}}) 
\quad \text{ if }\mathfrak{s}\in \mathfrak{S}_{r}^{m-k}
.\end{align}
Here $ x=(x_1,\ldots,x_k)\in \Srk $, 
$ \mathfrak{s}= 
\sum_{j=k+1}^m \delta _{x_j} \in \mathfrak{S}_{r}^{m-k}$, 
and $ \SSSm  $ is the set 
consisting of the permutations of 
$ (1,\ldots,m) $. 

If $ h = \hsym \in \dnuirD $, 
then one can regard $ h $ as $ h \in \dmirD  $, 
and we denote it by $ h ^{0} $. 
Indeed, $ h ^{0} $ is defined by
$ h ^{0}(\sum_{j=1}^{k}
\delta _{x_j}+\mathfrak{s}):=
h (x , \mathfrak{s}) $ on 
$ \sum_{m=k}^{\infty}\mathfrak{S}_{r}^m $
and by $ h ^{0}= 0  $ if 
$ \mathfrak{s}\not\in \sum_{m=k}^{\infty}\mathfrak{S}_{r}^m $. 
We remark $ \hsym ^{0}\circ \kappa = \hsym  $ by construction. 

Let $ h _{1}$ and $h _{2} \in \dnuirD $. 
Assume $ h _{2,\mathrm{sym}}=h _{2} $. 
Then we have 
\begin{align}\label{:31e}&
\int _{{S}^{k}\ts \mathfrak{S}}h _{1} h _{2} d\nuk = 
\int _{{S}^{k}\ts \mathfrak{S}}h _{1,\mathrm{sym}} h _{2} 
d\nuk =
\int _{\mathfrak{S}}h _{1,\mathrm{sym}} ^{0} h _{2} ^{0} d\murk 
\\ \label{:31h} &
\E ^{\nuk }(h _{1},h _{2})=
\E ^{\nuk }(h _{1,\mathrm{sym}},h _{2})=
\E ^{\murk }(h _{1,\mathrm{sym}}^{0},h _{2} ^{0})
.\end{align}
Let us take $ h _{1}=h $ and 
$ h _{2} = h _{\mathrm{sym}} $ in \eqref{:31h}. 
Then we have  
\begin{align}\label{:3i}&
\E ^{\nuk }(h ,h _{\mathrm{sym}})=
\E ^{\nuk }(h _{\mathrm{sym}},h _{\mathrm{sym}})=
\E ^{\murk }(h _{\mathrm{sym}} ^{0},h _{\mathrm{sym}} ^{0})
.\end{align}
Applying Schwarz's inequality 
to the first equality of \eqref{:3i} yields 
$$ 
\E ^{\nuk }(h ,h ) \ge 
\E ^{\nuk }
(h _{\mathrm{sym}},h _{\mathrm{sym}}). 
$$
Hence we can define 
$ h _{\mathrm{sym}} $ not only for $ h \in \dnuirD $ but also 
for $ h \in \dnurD $ as the limit of the 
$ \{ \E ^{\nuk }_{1}\}^{1/2}$-norm. 
Moreover, by \eqref{:31e} and \eqref{:3i} we have 
\begin{align}\label{:3j}&
\{ h _{\mathrm{sym}} ^{0} ; h \in \dnurD \} = 
\dmrrD 
.\end{align}
Similarly as \eqref{:3j} we have 
\begin{align}\label{:3n}&
\{ h _{\mathrm{sym}} ^{0} ; h \in \dnurDN \} = \dmrrDN 
.\end{align}
Since 
$\mu (\cdot )\le \murk (\cdot )\le N\mu (\cdot )$ 
on $ \sum _{m=1}^{N}\mathfrak{S}_{r,m}$ by 
\eqref{:30a}, we obtain 
\begin{align}\label{:3p}&
\dmrDN = \dmrrDN 
.\end{align}

\section{Identities 
among $ k $-labeled diffusions. } 
\label{s:4}
In this section we assume 
\thetag{M.1.0}, \thetag{M.1.$ k $}, \thetag{M.2} 
and \thetag{M.3}. 
The purpose of this section is to prove 
the identity between the diffusions 
associated with 
the Dirichlet spaces $ \7 $ and 
$ \9 $ introduced in \sref{s:2}. 
This identity is a key to 
the proof of \tref{l:26}. 

\begin{lem} \label{l:41}
Let $ \mathfrak{A}\subset {S}^{k}\ts \mathfrak{S}$ be such that 
$ \kappa ^{-1}(\kappa (\mathfrak{A}))= \mathfrak{A}$. Then 
$ \mathrm{Cap}^{\mu }(\kappa (\mathfrak{A}))= 0 $ implies 
$ \mathrm{Cap}^{\nuk }(\mathfrak{A})= 0 $. 
Here we regard $ \kappa $ as 
$ \map{\kappa }{{S}^{k}\ts \mathfrak{S}}{\mathfrak{S}} $ 
.\end{lem}
\begin{proof}
Without loss of the generality we can and do assume 
$ \mathfrak{A}\subset {S}^{k}_{r-1} \ts \mathfrak{S}$ and 
$ \mathfrak{A}\cap ({S}^{k}\ts \partial \mathfrak{S}_{r})= \emptyset $ 
for some $ r \in \N $.  
Since the capacity of a set $ B $ is given by the infimum of the capacity of the open sets including $ B $, we can assume without loss of generality that 
$ \kappa (\mathfrak{A}) $ is an open set. 
Then $ \mathfrak{A}$ becomes an open set. So by definition we have 
\begin{align}\label{:41}&
\mathrm{Cap}^{\mu }(\kappa (\mathfrak{A}))=
\inf \{ \E ^{\mu}_1(f,f); f\in \mathcal{D}^{\mu},\ 
f\ge 1 \text{ $ \mu $-a.e.\ on }\kappa (\mathfrak{A})\} 
,\\ \label{:41a} &
\mathrm{Cap}^{\nuk }(\mathfrak{A})=
\inf \{ \E ^{\nuk }_1(g,g); g \in \mathcal{D}^{\nuk },\ 
g \ge 1 \text{ $ \nuk $-a.e.\ on } \mathfrak{A}\} 
.\end{align}
Here $ \E ^{\mu}_1(f,f)= \E ^{\mu}(f,f)+
(f,f)_{\Lm } $ as usual and 
we set $ \E ^{\nuk }_1 $ similarly.

Since 
$ \mathfrak{A}\subset {S}^{k}_{r-1} \ts \mathfrak{S}$ and 
$ \mathfrak{A}\cap ({S}^{k}\ts \partial \mathfrak{S}_{r})= \emptyset $, 
we deduce that 
\begin{align}\label{:41b}&
\mathrm{Cap}^{\mu }(\kappa (\mathfrak{A}))=
\inf \{ \E ^{\mu}_1(f,f); f\in \dmirD ,\ 
f\ge 1 \text{ $ \mu $-a.e.\ on }\kappa (\mathfrak{A})\} 
.\end{align}
If $ f \in \dmirD $, then 
$ f \circ \kappa \in \mathcal{D}^{\nuk } $. Combining this with \eqref{:41}--\eqref{:41b} and 
the assumption 
$ \mathrm{Cap}^{\mu }(\kappa (\mathfrak{A}))= 0 $ 
completes the proof. 
\end{proof}

We consider parts of 
$ \PP ^{\mu }$ and $ \PP ^{\nuk }$. 
We refer to \cite{fot} for the definition of 
a part of Dirichlet space and related results.

Let $ \partial S _{r}=\{ |x|=r \}$ and 
$ \partial \mathfrak{S}_{r}= \{ \mathfrak{s}\in \mathfrak{S}; \ 
\mathfrak{s}(\partial S _{r})\ge 1  \}$ as before. 
Let 
\begin{align}\label{:41e}
\szr (\mathfrak{X})&=\inf \{ t>0 ; 
\mathfrak{X}_t\in \partial \mathfrak{S}_{r}\} 
\\\label{:42ee}
\szrN (\mathfrak{X})&= 
\inf \{ t>0 ; 
\mathfrak{X}_t\in \partial \mathfrak{S}_{r}\text{ or }
\mathfrak{X}_t \not\in  \sum _{m=1}^{N}\mathfrak{S}_{r,m}
\}
\\\label{:41f}
\sor (\mathfrak{X}^{1})&=
\inf \{ t>0 ; X_t \not\in \Srk \text{ or } 
\mathfrak{X}_t\in \partial \mathfrak{S}_{r}\} 
\\\label{:42o}
\sorN (\mathfrak{X}^{1})&=
\inf \{ t>0 ; X_t \not\in \Srk \text{ or } 
\mathfrak{X}_t\in \partial \mathfrak{S}_{r}\text{ or } 
\mathfrak{X}_t \not\in \sum _{m=0}^{N-k}\mathfrak{S}_{r,m}\ \} 
,\end{align}
where 
$ \mathfrak{X}\in C([0,\infty);\mathfrak{S})$ 
and 
$ \mathfrak{X}^{1}=(X,\mathfrak{X})
\in C([0,\infty);{S}^{k}\ts \mathfrak{S})$. 

Let $ \mathfrak{X}^{\szr }=
\{ \mathfrak{X}_{t\wedge \szr } \}$ and 
define $ \mathfrak{X}^{\szrN } $, 
$ \mathfrak{X}^{1,\sor } $, and 
$ \mathfrak{X}^{1,\sorN } $ in a similar fashion.  
Let $ \PP ^{\mu ,\szr } = 
(\{\PP ^{\mu }_{\mathfrak{s}}\}_{\mathfrak{s}\in\mathfrak{S}}
,\mathfrak{X}^{\szr })$. Set 
$ \PP ^{\mu ,\szrN }$, 
$ \PP ^{\nuk ,\sor }$, and 
$ \PP ^{\nuk ,\sorN }$ similarly. 
Then 
$ \PP ^{\mu ,\szr }$, 
$ \PP ^{\mu ,\szrN }$, 
$ \PP ^{\nuk ,\sor }$, and 
$ \PP ^{\nuk ,\sorN }$ are diffusions associated with the Dirichlet spaces 
$(\Em , \dmrD ,\Lm )$, $(\Em , \dmrDN ,\Lm )$, 
$ (\Enuk , \dnurD ,\Lnuk )$, 
and 
$ (\Enuk , \dnurDN ,\Lnuk )$, respectively.  
Let 
$ \PP ^{\mu ,\szr }_{\mathfrak{s}}= 
\PP ^{\mu }_{\mathfrak{s}}(\mathfrak{X}^{\szr }\in \cdot )$.  We set 
$ \PP ^{\mu ,\szrN }_{\mathfrak{s}}$, 
$ \Pxt ^{\nuk ,\sor }$, 
and 
$ \Pxt ^{\nuk ,\sorN }$ similarly. 
We note that these are the distributions of 
$ \PP ^{\mu ,\szr }$, 
$ \PP ^{\mu ,\szrN }$, 
$ \PP ^{\nuk ,\sor }$, and 
$ \PP ^{\nuk ,\sorN }$, respectively. 
Let 
$ \mathrm{Cap}^{\mu ,\szr } $ and  
$ \mathrm{Cap}^{\mu ,\szrN } $ 
be the capacities of 
$ \PP ^{\mu ,\szr }$ and  
$ \PP ^{\mu ,\szrN }$, respectively. 
\begin{lem} \label{l:42} 
Assume 
\thetag{M.1.0}, \thetag{M.1.$ k $}, 
\thetag{M.2}, and \thetag{M.3}. 
Then there exists 
$ \Akr \subset \Srk \ts \mathfrak{S}$ such that 
\begin{align}\label{:42a}&
 \kappa ^{-1}(\kappa (\Akr ))= \Akr ,\ 
\kappa (\Akr )\subset \mathfrak{S}_{\mathrm{single}}
,\\\label{:42d}&
\mathrm{Cap}^{\mu ,\szr }
(\kappa (\Srk \ts \mathfrak{S})\backslash 
\kappa (\Akr ))= 0 
,\\\label{:42zz}&
\PP ^{\mu ,\szr } _{\kappa (x,\mathfrak{s})}
(\mathfrak{X}_t\in \kappa (\Akr ) 
\text{ for all } t)= 1 
\quad \text{ for all }
(x,\mathfrak{s})\in \Akr 
,\\\label{:42z}&
\PP ^{\mu ,\szr } _{\kappa (x,\mathfrak{s})}
= 
\Pxt ^{\nuk ,\sor } \circ \kappa ^{-1} 
\quad \text{ for all }
(x,\mathfrak{s})\in \Akr 
.\end{align}
\end{lem}
\begin{proof} 
If for each $ N\in \N $ there exists a set 
$\AkrN \subset \Srk \ts \sum _{m=0}^{N-k}\mathfrak{S}_{r,m}$ such that 
\begin{align} \label{:42yyy}&
\kappa ^{-1}(\kappa (\AkrN ))= \AkrN 
,\ 
\kappa (\AkrN )\subset \mathfrak{S}_{\mathrm{single}}
,\\\label{:42y}&
\mathrm{Cap}^{\mu ,\szrN }
(\kappa (\Srk \ts \sum _{m=0}^{N-k}\mathfrak{S}_{r,m})\backslash \kappa (\AkrN )) = 0 
,\\\label{:42y]}&
\PP ^{\mu ,\szrN } _{\kappa (x,\mathfrak{s})}
(\mathfrak{X}_t\in \kappa (\AkrN ) 
\text{ for all } t)= 1 
\quad \text{ for all }
(x,\mathfrak{s})\in \AkrN 
,\\\label{:42xxx}&
\PP ^{\mu ,\szrN } _{\kappa (x,\mathfrak{s})} = 
\Pxt ^{\nuk ,\sorN } \circ \kappa ^{-1} 
\quad \text{ for all } (x,\mathfrak{s})\in \AkrN 
,\end{align}
then $ \Akr := \liminf_{N\to\infty }\AkrN $ 
satisfies \eqref{:42a}--\eqref{:42z}. 
Hence it only remains to prove 
such an $ \AkrN $ exists for each $ N $. 

Recall that $ \PP ^{\nuk ,\sorN } $ 
is the diffusion associated with 
$ (\Enuk , \dnurDN ,\Lnuk )$. 
Let 
$ T_{r,\mathrm{D},t}^{\nuk ,N} $ 
be the semigroup associated with 
$ \PP ^{\nuk ,\sorN } $. 
Then for $ f $ and $ g \in \dnurDN $ 
\begin{align}\label{:42f}&
\int _{{S}^{k}\ts \mathfrak{S}}
T_{r,\mathrm{D},t}^{\nuk ,N} 
f \cdot g  d\nuk -
\int _{{S}^{k}\ts \mathfrak{S}}
f \cdot g  d\nuk + 
\int_0^{t} \Enuk (T_{r,\mathrm{D},u}^{\nuk ,N}
f ,g )du 
= 0
.\end{align}
Now suppose $ g _{\mathrm{sym}}=g  $. 
Then by \eqref{:31e} and \eqref{:31h} we have 
\begin{align}\label{:42g}&
\int _{\mathfrak{S}}
(T_{r,\mathrm{D},t}^{\nuk ,N} f )_{\mathrm{sym}} ^{0}
\cdot g  ^{0} d\murk -
\int _{\mathfrak{S}}
f _{\mathrm{sym}} ^{0}\cdot g  ^{0} d\murk 
+ \int_0^{t}\E ^{\murk } (
(T_{r,\mathrm{D},u}^{\nuk ,N}f )
_{\mathrm{sym}} ^{0},g  ^{0}) du =0
.\end{align}
Let $ T_{r,\mathrm{D},t}^{\murk ,N}$ 
be the semigroup associated with 
$(\E ^{\murk }, \dmrrDN )$ on $\Lmr $. 
Then by \eqref{:3n} and \eqref{:42g} we have 
\begin{align}\label{:42v}&
T_{r,\mathrm{D},t}^{\murk ,N} 
(f _{\mathrm{sym}} ^{0})= 
(T_{r,\mathrm{D},t}^{\nuk ,N}f )_{\mathrm{sym}} ^{0}
.\end{align}
Let $ \mathrm{Cap}^{\murk ,\szrN } $ and 
$ \PP ^{\murk ,\szrN }  $ 
be the capacity and the diffusion associated with 
the Dirichlet space 
$(\E ^{\murk }, \dmrrDN ,\Lmr )$, respectively. 
Then by \eqref{:42v} together with \thetag{M.3}, 
we deduce that there exists 
$ \AkrN \subset \Srk \ts \sum _{m=0}^{N-k}\mathfrak{S}_{r,m}$ 
satisfying \eqref{:42yyy} and 
\begin{align}\label{:42yy}&
\mathrm{Cap}^{\murk ,\szrN }
(\kappa (\Srk \ts \sum _{m=0}^{N-k}\mathfrak{S}_{r,m})\backslash 
\kappa (\AkrN ))= 0 
\\ \label{:42y]]} &
\PP ^{\murk ,\szrN } _{\kappa (x,\mathfrak{s})}
(\mathfrak{X}_t\in \kappa (\AkrN ) 
\text{ for all } t)= 1 
\quad \text{ for all }
(x,\mathfrak{s})\in \AkrN 
,\\\label{:42w}&
\PP ^{\murk ,\szrN } _{\kappa (x,\mathfrak{s})} 
= 
\Pxt ^{\nuk ,\sorN } \circ \kappa ^{-1} 
\quad \text{ for all }
(x,\mathfrak{s})\in \AkrN 
. \end{align}

Recall that the diffuions 
$ \PP ^{\mu ,\szrN }  $ and 
$ \PP ^{\murk ,\szrN } $ 
in \eqref{:42xxx} and \eqref{:42w} 
are associated with the Dirichlet spaces 
$(\Em , \dmrDN ,\Lm )$ 
and 
$(\E ^{\murk }, \dmrrDN ,\Lmr )$, respectively. 
Note that $ \dmrDN = \dmrrDN $ by \eqref{:3p}. 
Moreover, these two  Dirichlet spaces have the common state space $ \sum _{m=1}^{N}\mathfrak{S}_{r,m} $. 
On each connected component $ \{\mathfrak{S}_{r,m}\}$ of the state space, the measures $ \mu $ and $ \murk $ 
are constant multiplication of each other. 
Hence the associated diffusions 
are the same until they hit the boundary. 
Since these Dirichlet forms enjoy 
the Dirichlet boundary conditions, 
we see that eventually 
these two Dirichlet spaces define 
the same diffusion. 
This combined with \eqref{:42yy}--\eqref{:42w} 
we obtain \eqref{:42y}--\eqref{:42xxx}, 
respectively. 

We therefore deduce that $ \AkrN $ $(N \in \N )$ 
satisfy \eqref{:42yyy}--\eqref{:42xxx}, which completes the proof of \lref{l:42}. 
\end{proof}

Let $ r(i)= r $ if $ i $ is odd, and 
$ r(i)=r+1 $ if $ i $ is even. 
Let for $ i \ge 2 $ 
\begin{align}\label{:43e}&
\szir (\mathfrak{X})= 
\inf \{ t>\sziir ; \mathfrak{X}_t\in 
\partial \mathfrak{S}_{r(i)} \} 
\\\label{:43f}&
\soir (\mathfrak{X}^{1})= 
\inf \{ t>\soiir ; X_t\in \partial \Sri  
\text{ or } 
 \mathfrak{X}_t \in \partial \mathfrak{S}_{r(i)} \} 
,\end{align}
where we set 
$ \bar{\sigma } ^{a}_{1}=\sigma ^{a}_r $ 
($ a=0,1 $). 
For $ \mathfrak{X}=
\{ \sum_{i}\delta _{X_t^i}\}
\in C([0,\infty);\mathfrak{S})$ satisfying 
$\mathfrak{X}_t \in \mathfrak{S}_{\mathrm{single}}$ 
for all $ t $ and 
$ \mathfrak{X}_0=\sum_{i}\delta _{x_i} $, 
we choose the first $ k $-particles 
$$ \{\bar{X}_t\} = 
\{ (X_t^{1},\ldots,X_t^{k}) \} 
\in C([0,\infty);{S}^{k})$$
such that 
$ \bar{X}_0 = x =(x_1,\ldots,x_k)$ 
and that 
$ \mathfrak{X}_t= \sum_{j=1}^k\delta _{X_t^{j}}+ 
\sum_{j>k}\delta _{X_t^{j}}$. 

Let 
$ \Omegaiz =\{ \omega \, ;\, 
\bar{X}_{\szir }(\omega )\in \Srr \}$ and  
$ \Omegaione =\{ \omega \, ;\, 
X_{\soir }(\omega )\in \Srr \}$. Let 
\begin{align}\label{:4a}&
\Omegak = \bigcap_{i=1}^{\infty }\Omegaik ,\quad 
\bar{\sigma }^{a}_{\infty }=\limi{i} \sair 
\quad (a=0,1).
\end{align}
Since 
$ X_{\soir }\in \Srr $ on $ \Omegaione $, 
$ \Srr \subset {S}^{k}_{r+1} $, and 
$ {S}^{k}_{r+1}\cap \partial {S}^{k}_{r+1} 
= \emptyset $, we deduce that 
\begin{align}\label{:4b}&
X_t(\omega )\in {S}^{k}_{r+1} 
\quad \text{ for all }
0 \le t < \bar{\sigma }^{1}_{\infty } (\omega )
,\ \text{ for all }
\omega \in \Omegaone 
.\end{align}
\begin{lem} \label{l:43} 
Assume 
\thetag{M.1.0}, \thetag{M.1.$ k $}, 
\thetag{M.2}, and \thetag{M.3}. 
Let $ \Akr $ be as in \lref{l:42}. 
Then for all $ (x,\mathfrak{s}) \in \Akr $ 
the following holds. 
\\ \thetag{1} 
$ \bar{\sigma }^{0}_{\infty }
=\infty $ for 
$ \Pm _{\kappa (x,\mathfrak{s})}
(\cdot \ ; \Omegaz )
\text{-a.e.\ }\omega $, 
\\
\thetag{2} 
$ \bar{\sigma }^{1}_{\infty }=\infty $ for 
$ \Pxtnu (\cdot \ ; \Omegaone )
\text{-a.e.\ }\omega  $. 
\end{lem}
\begin{proof}
Let 
$ \partial \mathfrak{S}_{r}=\{ \mathfrak{s}\, ;\, 
\mathfrak{s}(\partial S _{r})\ge 1 \} $ as before. 
Then by the continuity of the sample paths, 
\eqref{:43e} and \eqref{:4a}, we deduce 
\begin{align}\label{:43a}&
\mathfrak{X}_{\bar{\sigma }^{0}_{\infty }} 
= \limi{i}\mathfrak{X}_{\bar{\sigma }^{0}_{i} }
\in \partial \mathfrak{S}_{r}\cap \partial \mathfrak{S}_{r+1}   
\quad \text{on }
\{ \bar{\sigma }^{0}_{\infty }<\infty \} 
.\end{align}

Suppose  
$\Pm _{\kappa (x,\mathfrak{s})}
(\bar{\sigma }^{0}_{\infty }<\infty \ 
;\Omegaz)>0 $. Then  by \eqref{:43a} we have 
\begin{align}\label{:43k}&
\Pm _{\kappa (x,\mathfrak{s})} 
(\mathfrak{X}_{\bar{\sigma }^{0}_{\infty }} 
\in \partial \mathfrak{S}_{r}\cap \partial \mathfrak{S}_{r+1}  
 \ ; \Omegaz )>0 
.\end{align}
Hence $ \int_{\mathfrak{S}}\Pm _{\mathfrak{s}}
(\sigma _{\partial \mathfrak{S}_{r}\cap \partial \mathfrak{S}_{r+1}}
< \infty )\mu (d\mathfrak{s}) > 0 $, where  
$ \sigma _{\partial \mathfrak{S}_{r}\cap \partial \mathfrak{S}_{r+1}} $ is the first hitting time to the set 
$\partial \mathfrak{S}_{r}\cap \partial \mathfrak{S}_{r+1}$. 
By the general theory of Dirichlet forms 
(see \cite[Theorem 4.2.1. (ii)]{fot}) 
it follows from this that 
\begin{align}\label{:43b}&
\text{Cap}^{\mu }
(\partial \mathfrak{S}_{r}\cap \partial \mathfrak{S}_{r+1} )> 0 
.\end{align}

On the other hand, since the $ n $-correlation functions $ \rho ^n $ 
of $ \mu $ are locally bounded by \thetag{M.2}, 
it is not difficult to see that 
$ \text{Cap}^{\mu }
(\partial \mathfrak{S}_{r}\cap \partial \mathfrak{S}_{r+1} )=0 $. 
This contradicts \eqref{:43b}.  
Hence we obtain 
$\Pm _{\kappa (x,\mathfrak{s})}
(\bar{\sigma }^{0}_{\infty }<\infty \ 
;\Omegaz)= 0 $, which implies \thetag{1}. 
The proof of \thetag{2} is similar to that of \thetag{1}. 
\end{proof}

Let 
\begin{align}\label{:44e}&
\tau ^{0}_{r,x} (\mathfrak{X})= 
\inf \{ t>0 ; \bar{X}_t \in \partial \Srk \} 
,\\\label{:44f}&
\tau ^{1}_r (\mathfrak{X}^{1})= 
\inf \{ t>0 ; X_t\in \partial \Srk \} 
.\end{align}
\begin{rem}\label{r:41}
The stopping times 
$ \szir ,\ \soir $ and $ \ \tau ^{1}_r  $ 
are the hitting times 
to the subsets of the state spaces. 
So one can relate the stopped processes 
to the parts of Dirichlet forms. 
However, $ \tau ^{0}_{r,x} $ is not a hitting time 
to any subset in the state space $ \mathfrak{S}$. 
So one can not relate 
the associated stopped process 
to a part of the Dirichlet form, 
which is the reason we prepare 
\lref{l:42} before \lref{l:44}. 
\end{rem}

\begin{lem} \label{l:44} 
Assume 
\thetag{M.1.0}, \thetag{M.1.$ k $}, 
\thetag{M.2}, and \thetag{M.3}. 
Let $ \Akr $ be as in \lref{l:42}. 
Then 
\begin{align}\label{:44}&
\PP ^{\mu ,\tauzrx  } _{\kappa (x,\mathfrak{s})} = 
\Pxt ^{\nuk ,\tauor  } \circ \kappa ^{-1} 
\quad \text{ for all }(x,\mathfrak{s})
\in \Akr \cap \Akrr 
.\end{align}
Let $ \tau ^{0}_{\infty ,x}=
\limi{r}\tau ^{0}_{r,x} $ and 
$ \tau ^{1}_{\infty }=\limi{r}\tau ^{1}_r $. 
Then 
\begin{align}\label{:44z}&
\PP ^{\mu ,\tau ^{0}_{\infty ,x} } 
_{\kappa (x,\mathfrak{s})}
= 
\Pxt ^{\nuk ,\tau ^{1}_{\infty } } \circ \kappa ^{-1} 
\quad \text{ for all }(x,\mathfrak{s}) 
\in \liminf _{r\to\infty }\Akr 
.\end{align}
\end{lem}
\begin{proof}
Suppose $ \omega \in \Omegaone $. 
Then by \eqref{:4b} and \lref{l:43} we have 
$ \mathfrak{X}^{1}_t \in {S}^{k}_{r+1} \ts \mathfrak{S}$ 
for all $ 0\le t < \infty $. 
In particular, $ X _t \in {S}^{k}_{r+1} $ 
for all $ 0\le t < \infty $. 
Hence by using \lref{l:42} with $ r $ and 
$ r+1 $ combined with 
the strong Markov property repeatedly, 
we obtain for all 
$ (x,\mathfrak{s})\in \Akr \cap \Akrr $ 
\begin{align}\label{:44a}&
\PP ^{\mu ,\szir } _{\kappa (x,\mathfrak{s})} 
(\ \cdot  \ ; \Omegaz  ) 
= 
\Pxt ^{\nuk ,\soir }(\ \cdot  \ ; \Omegaone  ) \circ \kappa ^{-1} 
\quad \text{ for all } i 
.\end{align}
Hence by \lref{l:43} we have 
\begin{align}\label{:44b}&
\Pm _{\kappa (x,\mathfrak{s})} 
(\ \cdot  \ ; \Omegaz ) 
= 
\Pxtnu 
(\ \cdot  \ ; \Omegaone ) \circ \kappa ^{-1} 
\quad \text{ for all }
(x,\mathfrak{s})\in \Akr \cap \Akrr 
.\end{align}

Next suppose $ \omega \not\in \Omegaone $. 
Then there exists an $ i $ such that 
$ X_{\soir } \not\in \Srr $ 
and 
$ X_{\sojr } \in \Srr $ 
for all $ j< i $. 
Let $ \Omegaionestar  $ denote the collection of 
such $ \omega $ : 
\begin{align}\notag &
\Omegaionestar =\{ \omega \, ;\, 
X_{\soir }(\omega )\not\in \Srr , \ 
X_{\sojr }(\omega )
\in \Srr \ (\forall j < i) \}
.\end{align}
By \lref{l:42} and 
the strong Markov property we have 
\begin{align}\label{:44c}&
\PP ^{\mu ,\szir } _{\kappa (x,\mathfrak{s})}
 (\ \cdot  \ ; \Omegaizstar  ) 
= 
\Pxt ^{\nuk ,\soir }
(\ \cdot  \ ; \Omegaionestar  ) 
\circ \kappa ^{-1} 
\quad \text{ for all }
(x,\mathfrak{s})\in \Akr \cap \Akrr 
.\end{align}
By construction 
$ \tau ^{a}_{r,x} \le \sair $ ($ a=0,1 $). 
Hence \eqref{:44c} implies 
\begin{align}\label{:44d}&
\PP ^{\mu ,\tauzrx } _{\kappa (x,\mathfrak{s})}
 (\ \cdot  \ ; \Omegaizstar  ) 
= 
\Pxt ^{\nuk ,\tauor }
(\ \cdot  \ ; \Omegaionestar  ) 
\circ \kappa ^{-1} 
\quad \text{ for all }
(x,\mathfrak{s})\in \Akr \cap \Akrr 
.\end{align}

We now see that 
$ \Omega ^{a}= \Omegaa +
\sum_{i=1}^{\infty}\Omegaiastar $ ($ a=0,1 $). 
Hence \eqref{:44} follows from \eqref{:44b} and \eqref{:44d}. 
\eqref{:44z} follows from \eqref{:44} immediately. 
\end{proof}

{\em Proof of \tref{l:24}. }
Let $ \tilde{\mathfrak{S}} = 
\cap_{k\in\N }\{\liminf _{r\to\infty }\Akr \} $. 
Then by \eqref{:42d} we have \eqref{:24a}. 
Moreover, by \eqref{:24c} we deduce that 
$ \tau ^{0}_{\infty ,x}= \infty $ 
for $ \Pm _{\mathfrak{s}} $-a.s.\ 
for all 
$ \mathfrak{s}\in \tilde{\mathfrak{S}} $ such that $ \mathfrak{s}(x)= 1 $. 
Hence by \eqref{:44z} of \lref{l:44} 
we obtain \eqref{:24b} and \eqref{:24}. 
\qed 

\section{Tagged particle processes } 
\label{s:5}
In this section we prove \tref{l:26} 
and \tref{l:27}. So we take $ {S}= \Rd $ and 
$ k=1 $. We set $ \nu = \nu ^{1} $. 
Let $ \iota $ 
be the transformation on $ \RdT $ defined by 
\begin{align}\label{:51z}&
\iota (x,\mathfrak{s})= (x,\vartheta _{x}(\mathfrak{s}))
.\end{align}
Then by \thetag{M.4} we deduce that 
\begin{align}\label{:51}&
\nu \circ \iota ^{-1} = \0 
.\end{align}
We regard $ \iota $ 
as the transformation on $ {C([0,\infty); \RdT )} $, 
denoted by the same symbol $ \iota $, by 
$ \iota (\mathfrak{X}^{1})=
\{\iota(\mathfrak{X}^{1}_t)\}$.

\begin{lem}   \label{l:51} 
Assume 
\thetag{M.1.0}, \thetag{M.1.1}, 
\thetag{M.2}--\thetag{M.4}. 
Then we have the following. 
\\
\thetag{1} The bilinear form 
$(\EXY ,\dXYi )$ is closable on $ \LXY $. 
\\
\thetag{2} Let 
$(\EXY ,\dXY )$ be the closure of $(\EXY ,\dXYi )$ on $ \LXY $. 
Let 
\begin{align}\label{:51a}&
\PP ^{XY }_{\xsss }=
\PP _{\iota ^{-1}(x,\mathfrak{s})}^{\nu } 
\circ \iota ^{-1}
.\end{align}
Then 
$\PP ^{XY } = 
(\{\PP ^{XY }_{\xsss }\}_{\xsss \in\RdT },\mathfrak{X}^{1})$ 
is a diffusion associated with the Dirichlet space 
$(\EXY ,\dXY , \LXY )$. 
\\
\thetag{3} The Dirichlet space 
$(\EXY ,\dXY ,\LXY )$ is quasi-regular. 
\end{lem}

\begin{proof}
By \eqref{:51} we have 
\begin{align}\label{:51c}&
(f \circ \iota ,g \circ \iota )_{\Lnu }=
(f ,g )_{\LXY }
.\end{align}
We next calculate the transformation of $ \Done $ 
under the change of coordinate induced by $ \iota $. 
By a straightforward calculation we see that 
\begin{align}\label{:51d}&
\Done [f \circ \iota ,g \circ \iota ] =
(\DXY [f ,g ])\circ \iota 
\quad \text{ for $ f , g \in \CD  $. }
\end{align}
By \eqref{:51} and \eqref{:51d} 
we obtain the isometry of the bilinear forms 
$ (\Enu , \dnui ) $ and $ (\EXY , \dXYi ) $
under the transformation induced by $ \iota $. 
Indeed, 
the map $ \map{\iota ^{*}}{\dXYi }{\dnui }$ 
defined by $ \iota ^{*}(f )= f \circ \iota  $ 
is bijective and 
\begin{align}\label{:51b}&
\Enu (f \circ \iota ,g \circ \iota )
=\EXY (f ,g )
.\end{align}
By \eqref{:51c} and \eqref{:51b} 
the closability of 
 $ (\EXY , \dXYi ) $ on 
$ \LXY $ follows from that of 
$ (\Enu , \dnui ) $ on $ \Lnu $, 
which is given by \thetag{M.1.1}. 
We have thus proved \thetag{1}.

Since $ \iota $ is the transformation on $ \RdT $, it is clear that 
$  \PP ^{XY }$ 
is a diffusion with state space $ \RdT $. 
Recall that 
$ \PP ^{XY }_{\xsss }=
\PP _{\iota ^{-1}(x,\mathfrak{s})}^{\nu } 
\circ \iota ^{-1 } $ 
and that 
$ \PP ^{\nu }$ is the diffusion associated with 
$(\Enu , \dnu ,\Lnu )$. 
By \eqref{:51c} and \eqref{:51b} 
the Dirichlet spaces 
$(\Enu , \dnu ,\Lnu )$ and 
$(\EXY ,\dXY ,\LXY )$ are isometric. 
Hence we conclude 
$ \{ \PP ^{XY }_{\xsss } \}$ 
is associated with the Dirichlet space 
$(\EXY ,\dXY , \LXY )$. 

By the theorem due to Albeverio-Ma-R\"ockner 
(see \cite[Theorem 5.1]{mr}), 
the quasi-regularity of the Dirichlet space follows from 
the existence of the associated diffusion. 
Hence \thetag{3} follows from \thetag{2} immediately. 
\end{proof}

\begin{lem} \label{l:52} 
Let 
$ \mathrm{Cap}^{XY} $ be the capacity associated 
with the Dirichlet space $(\EXY ,\dXY ,\LXY )$.
Let $ \PP ^{XY } $ be the associated diffusion 
as in \lref{l:51}. 
Then there exists a subset 
$ \Xi \subset \RdT $ such that 
\begin{align}\label{:51u}&
\PP ^{XY }_{\xsss }(\mathfrak{X}\in \cdot )=
\PP ^{XY }_{\ysss } (\mathfrak{X}\in \cdot )
\quad \text{ for all }(x, \mathfrak{s}),\ (y,\mathfrak{s}) \in \Xi  ,
\\\label{:52z}&
\mathrm{Cap}^{XY}(\Xi ^{c})= 0
.\end{align}
Here we set $ \mathfrak{X}^{1}= (X,\mathfrak{X}) \in 
C([0,\infty);{S}\ts \mathfrak{S}) $ as before. 
\end{lem}
\begin{proof}
It is clear that for each $ a\in \Rd $ 
\begin{align}&\notag 
(f (\cdot - a , *),g (\cdot - a , *))_{\LXY }
=(f ,g )_{\LXY }
\\ \notag &
\EXY (f (\cdot - a , *),g (\cdot - a , *))=
\EXY (f ,g )
.\end{align}
Hence we see that 
the equality in \eqref{:51u} holds for a.e.\ 
$ (x, \mathfrak{s}),\ (y,\mathfrak{s}) \in \RdT $. 

We next strength the equality in \eqref{:51u} from {\em a.e.\ }to {\em all } 
on $ \Xi $ for some $ \Xi $ satisfying 
$ \mathrm{Cap}^{XY}(\Xi ^{c})=0  $.

For each Borel set $ \mathfrak{A} $ of the form 
$ \mathfrak{A}=\{ \mathfrak{X}_{t_1}\in A_{1},\ldots,
\mathfrak{X}_{t_i}\in A_{i}  \}$, where 
$ A_j \in \mathcal{B}(\mathfrak{S})$ \ ($ j=1,\ldots,i $),  
we see that 
$ \PP ^{XY }_{\xsss }(\mathfrak{X}\in \mathfrak{A}) $ 
is quasi-continuous in $ (x,\mathfrak{s}) $. 
Hence there exists a subset 
$ \Xi \subset \RdT $ such that 
$ \mathrm{Cap}^{XY}(\Xi ^{c})=0  $ and that 
$ \Xi = \cup_{n=1}^{\infty}K_n $ 
for some increasing sequence of 
closed set and, moreover, 
the restriction of 
$ \PP ^{XY }_{\xsss }(\mathfrak{X}\in A) $ on 
$ K_n $ is continuous in $ (x,\mathfrak{s}) $ for all $ n $. 
This means, with a help of the monotone class theorem, 
\eqref{:51u} holds for $ \Xi $ as above. 
\end{proof}

\begin{lem} \label{l:53} 
Assume 
\thetag{M.1.0}, \thetag{M.1.1}, 
\thetag{M.2}--\thetag{M.4}. 
Then we have the following. 
\\
\thetag{1} The bilinear form $(\EY ,\dYi )$ is 
closable on $ \Lmz $. 
\\
\thetag{2} Let $ \Xi $ be as in \lref{l:52}. 
Let 
$ \{\PPY _{\mathfrak{s}}\}_{\mathfrak{s}\in\mathfrak{S}}$ 
be the family of probability measures on 
$ C([0,\infty );\mathfrak{S}) $ defined by 
\begin{align*}&
\PPY _{\mathfrak{s}}= 
\PP ^{XY }_{\xsss }(\mathfrak{X} \in \cdot ) 
&& \text{ if  $ (x,\mathfrak{s})\in \Xi $ for some $ x\in\Rd $ 
,}
\\ \notag 
& \PPY _{\mathfrak{s}}(\mathfrak{X}_t=\mathfrak{s}\text{ for all }t )=1 
&&\text{ otherwise. }
\end{align*}
Then 
$ \PPY =
(\{\PPY _{\mathfrak{s}}\}_{\mathfrak{s}\in\mathfrak{S}},\mathfrak{X})$ 
is a diffusion. 
\end{lem}
\begin{proof}
Let $ \varphi \in \Cz $ and $ f \in \dYi $ . Then 
\begin{align}\label{:52c}&
\|\varphi \ot f \|_{\LXY }=
\|\varphi \|_{L^2(dx)} \|f \|_{\Lmz }
,\\ \label{:52e}&
\EXY 
(\varphi \ot f ,\varphi \ot f )=
\|\varphi \|^2_{L^2(dx)}  \EY (f ,f ) 
+ 
\frac{1}{2}\|\nabla \varphi \|^2_{L^2(dx)} 
 \|f \|^2_{\Lmz }
.\end{align}
Indeed, \eqref{:52c} is a straightforward calculation. 
As for \eqref{:52e} we see 
\begin{align}\label{:52d}
\DXY [\varphi \ot f ,\varphi \ot f ] =&
\varphi ^2 \ot \DY [f ,f ]+
\frac
{|\nabla \varphi |^2 \ot f ^2 }{2}  
-
(\varphi \nabla \varphi , f Df )_{\Rd } 
.\end{align}
Then integrating over $ \RdT $ 
by $ \0 $ and noticing 
$$ 
\int_{\RdT } 
(\varphi \nabla \varphi , f Df )_{\Rd } 
\0 = 
(\int_{\Rd }\varphi \nabla \varphi dx , 
\int _{\Theta }f Df d\muz )_{\Rd } = 0 
,$$
we obtain \eqref{:52e}.

By \eqref{:52c} and \eqref{:52e} the closability of 
$(\EY ,\dYi )$ on $ \Lmz $ follows from the one 
of $(\EXY ,\dXYi )$ on $ \LXY $, which has been already obtained in \lref{l:51} \thetag{1}. 
We thus prove \thetag{1}. 

We next prove \thetag{2}. 
By \eqref{:51u} we see that 
for any $ A \in \mathcal{B}(C([0,\infty); \mathfrak{S}) )$ 
\begin{align}\label{:52a} &
\PPY _{\mathfrak{s}} (\mathfrak{X}\in A  ) = 
\PP ^{XY }_{\xsss }((X ,\mathfrak{X}) \in \Czi \ts A )
\quad \text{for all $ (x,\mathfrak{s}) \in \Xi $}
.\end{align}
We remark  
$ \PP ^{XY }$ is a diffusion on $ \RdT $ 
and $ \Xi ^c $ is an exceptional set, that is, 
$ \PP ^{XY }_{\xsss  }(\sigma _{\Xi ^c}<\infty )=0 $ 
for q.e.\ $ (x,\mathfrak{s}) $ because of 
$ \mathrm{Cap}^{XY}(\Xi ^c)=0  $. 
Hence we deduce from \eqref{:52a} that 
$ \PPY $ is a diffusion with state space $ \mathfrak{S}$. 
\end{proof}

\begin{lem} \label{l:54} 
Let $(\EY ,\dY )$ be the closure of $(\EY ,\dYi )$ on $ \Lmz $. \\
\thetag{1} The diffusion 
$ \PPY $ in \lref{l:53} 
is associated with $ (\EY , \dY ) $ on $ \Lmz $.
\\\thetag{2} The Dirichlet form 
$(\EY ,\dY )$ on $ \Lmz  $ is quasi-regular. 
\end{lem}

\begin{proof}
Let $\EE ^{Y} _{\mathfrak{s}}$ denote the expectation with respect to $\PPY _{\mathfrak{s}}$. 
Let $ \{T^{Y}_{t}\} $ be the semigroup defined by 
$ T^{Y}_{t}f = 
\EE ^{Y} _{\mathfrak{s}}[f (\mathfrak{X}_t)] $. 
Let $ \{T^{XY}_{t}\} $ be the semigroup associated 
with the Dirichlet space 
$(\EXY ,\dXY ,\LXY )$. Then we deduce that 
\begin{align}\label{:52h}&
1\ot (T^{Y}_{t} f ) 
= T^{XY}_{t} (1\ot f )
.\end{align}

Let 
$ \rho (x)= \cref{;c51}(1+|x|^{2(d+4)})^{-1/2} $ 
such that 
$ \int \rho ^2dx = 1 $, where 
$ \Ct{;c51} $ is 
the normalizing constant. 
Let $ L^2(\rho )= L^2(\RdT , \rho ^2 \0 ) $ and 
\begin{align}\label{:54a}&
\EXY _{\rho , \lambda }(f ,g )=
\EXY (f , \rho ^2 g ) + 
\lambda (f ,g )_{L^2(\rho )}
.\end{align}
Then there exists $ \lambda _0 $ such that 
$ (\EXY _{\rho , \lambda } ,\dione )$ 
is positive and closable on 
$ L^2(\rho ) $ for all $\lambda > \lambda _0 $ 
(see \cite[Lemma 2.1]{o.inv2} for proof). 
We fix such a $ \lambda $ and denote by 
$ \{T^{\lambda }_t\} $ the semigroup associated 
with the closure 
$ (\EXY _{\rho , \lambda } ,\mathcal{D}_{\rho }^{XY})$ 
of 
$ (\EXY _{\rho , \lambda } ,\dione )$ 
on $ L^2(\rho ) $. 
It is known that (see \cite[234 p]{o.inv2})
\begin{align}\label{:54b}&
T^{XY}_{t} (1\ot f ) = 
e^{\lambda t} T^{\lambda }_t (1\ot f )  
.\end{align}
By a direct calculation we see that 
\begin{align}\label{:54c}&
\EXY _{\rho , \lambda } (1\ot f ,1\ot f )= 
 \EY (f ,f ) + \lambda (f ,f )_{\Lmz }
.\end{align}

Let $ \tilde{\mathcal{D}}^{Y} $ be 
the domain of the Dirichlet space associated with 
$ \{T^{Y}_{t}\} $ on $ \Lmz $. 
By \eqref{:52h} and \eqref{:54b} we obtain 
$ f \in \tilde{\mathcal{D}}^{Y} $  
if and only if 
$ 1\ot f \in \mathcal{D}_{\rho }^{XY} $. 
By \eqref{:54c} we see that 
 $ 1\ot f \in \mathcal{D}_{\rho }^{XY} $ 
if and only if $ f \in \dY $. 
Collecting these we obtain that 
$ \tilde{\mathcal{D}}^{Y}=\dY $. 

Let $ \psi \in \Cz $ be such that 
$ \int \psi \rho ^2 dx \not=0 $. 
Then we have for any $ f ,g \in \dY $ 
\begin{align}\label{:52k}&
\limi{\alpha  }\alpha  ^2 
( \int_0^{\infty }e^{-\alpha  t}
\{ 
1\ot f - T^{XY}_{t}(1\ot f ) \}dt ,
\ \psi \ot g \ )_{L^2(\rho )}
\\ \notag & 
= \int_{\Rd }\psi \rho ^2 dx \cdot \EY (f , g ) 
.\end{align}
By using \eqref{:52h} and \eqref{:52k} and then 
by dividing the both sides 
by $ \int_{\Rd }\psi \rho ^2 dx $, we obtain 
\begin{align}\label{:52j}&
\limi{\alpha  }\alpha  ^2 
( \int_0^{\infty }e^{-\alpha  t} 
\{ f - T^{Y}_{t}f \}\ dt , 
\ g  )_{\Lmz } =  \EY (f , g ) 
.\end{align}
This implies $ \{T^{Y}_{t}\} $ is the semigroup associated with  
the Dirichlet form $ (\EY ,\dY ) $ on $ \Lmz $ 
(see Lemma 1.3.4 in \cite{fot}). 
So we conclude $ \PPY $ is associated with 
$ (\EY , \dY ) $ on $ \Lmz $. 

\thetag{2} is immediate from \thetag{1} 
similarly as \lref{l:51}. 
\end{proof}

\noindent 
{\em Proof of \tref{l:26}. }
\thetag{1} follows from \lref{l:53} and 
\lref{l:54}. 
\thetag{2} follows from \lref{l:51}.
\qed

\noindent 
{\em Proof of \tref{l:27}. }
\eqref{:27a} follows from \tref{l:24} and \lref{l:51} \thetag{2}. 
\eqref{:27b} follows from \tref{l:24}, \lref{l:51} and 
\lref{l:53} immediately. 
\qed

\section{Non-explosion of tagged particles.}
\label{s:6}

Throughout this section we set 
$ \nur = \nur ^{1} $ and $ \mur = \mur ^{1} $. 
In this section we prove \tref{l:25}. 
By \eqref{:44z} in \lref{l:44} 
the non-explosion property of 
tagged particles follows from 
the conservativeness of the diffusion $ \PP ^{\nu }$. 
Then we apply a result in \cite{fot} to prove this as follows.

\begin{lem} \label{l:61} 
Assume 
\thetag{M.1.0}, \thetag{M.1.1}, 
\thetag{M.2}, and \thetag{M.3}. 
Assume \eqref{:25a}. 
Then $ \PP ^{\nu }$ is conservative. 
\end{lem}
\begin{proof}
Applying Theorem 5.7.2 in \cite{fot} to the diffusion $ \PP ^{\nu }$ yields \lref{l:61}. 
\end{proof}
We next prepare several notations 
used in the rest of this section. 

Let 
$ \mathfrak{X} \in 
C([0,\infty);\mathfrak{S}_{\mathrm{single}}) $. 
We write 
$ \mathfrak{X}=\{ \sum_i \delta_{X_t^i} \} $ and 
set $ X^i \in C(I_i ; \R^d)$. 
We take $ I_i $ to be the maximal interval. 
By construction we deduce that $ I_i $ is 
of the form $[0,b_i)$ or $(a_i,b_i)$. 
Let $ \mathbf{I} = \{ i; I_i =[0,b_i) \} $ 
and $ \mathbf{J} = \{ i; I_i =(a_i,b_i) \} $. 
Then 
$ \mathfrak{X}= 
\sum_{i\in \mathbf{I}} \delta_{X_t^i} +
\sum_{i\in \mathbf{J}} \delta_{X_t^i} =: 
\mathfrak{X}^{\mathbf{I}} + 
\mathfrak{X}^{\mathbf{J}}$. 

We relabel $ \mathfrak{X}^{\mathbf{I}} $  as 
$ \mathfrak{X}^{\mathbf{I}} =
\{\sum_x \delta_{X^{x}_t} \}$, 
where $ x \in {S}$ is such that $ X^{x}_0=x $. 
Let 
\begin{align}\label{:62c}&
\xi ^{x} (\mathfrak{X})= \inf \{ t > 0 \, ;\, 
\sup_{0\le s < t} |X^{x}_s|= \infty \} 
,\\\label{:62a}&
\xi _{r} (\mathfrak{X})= \inf \{ t > 0 \, ;\, 
\min_{|x|<r} \xi ^{x} (\mathfrak{X})<t \} 
\quad (r \in \N\cup \{ \infty \})
,\\\label{:62d}&
\mathfrak{A}_r = 
\{ \mathfrak{s}\in \mathfrak{S}\, ;\, 
\PP ^{\mu }_{\mathfrak{s}}(\xi _{r}<\infty )>0 \}
.\end{align}
For a path 
$ \mathfrak{X}^1 = (X,\mathfrak{X})  
$ 
we define the stopping time $ \eta $ by 
\begin{align}\label{:62b}&
\eta (\mathfrak{X}^1)= \inf \{ t > 0 \, ;\, 
\sup_{0\le s < t} |X_s| = \infty \} 
.\end{align}
\begin{lem} \label{l:62} 
Suppose 
$ \int _{{S}\ts \mathfrak{S}} 
\PP ^{\nu }_{(x,\mathfrak{s})}(\eta < \infty ) d\nur = 0 $. 
Then $ \mu (\mathfrak{A}_r) =0 $. 
\end{lem}
\begin{proof} 
Let $ \mur $ be as in \eqref{:!1}. 
For $ m\ge 1 $ let $ \Ct{;64a} $ 
be a constant such that 
$$ 
\mu (\cdot \cap \mathfrak{S}_{r,m} )\le 
\cref{;64a}
\mur (\cdot \cap \mathfrak{S}_{r,m}) 
.$$ 
Then we see that 
\begin{align}\label{:62e}
\int _{\mathfrak{S}_{r,m}} \PP ^{\mu }_{\mathfrak{s}}(\xi _{r}<\infty ) \mu (d\mathfrak{s}) 
 & \le 
\int _{\mathfrak{S}_{r,m}} \sum_{\mathfrak{s}(x)\ge 1,\ |x| < r}
\PP ^{\mu }_{\mathfrak{s}}(\xi ^{x}<\infty ) \mu (d\mathfrak{s}) 
\\ \notag & \le \cref{;64a}
\int _{\mathfrak{S}_{r,m}} \sum_{\mathfrak{s}(x)\ge 1,\ |x| < r}
\PP ^{\mu }_{\mathfrak{s}}(\xi ^{x}<\infty ) \mur (d\mathfrak{s}) 
\\ \notag & = \cref{;64a}
\int _{{S}\ts \mathfrak{S}_{r}^{m-1}} 
\PP ^{\mu }_{\kappa (x,\mathfrak{s})} (\xi ^{x}<\infty ) 
\nur (dxd\mathfrak{s})
\\ \notag & = \cref{;64a}
\int _{{S}\ts \mathfrak{S}_{r}^{m-1}} \PP ^{\nu }_{(x,\mathfrak{s})}(\eta <\infty ) 
\nur (dxd\mathfrak{s})
\quad \text{by \eqref{:44z}} 
.\end{align}
Hence we have 
$\int _{\mathfrak{S}_{r,m}}\PP ^{\mu }_{\mathfrak{s}}(\xi _{r}<\infty )\mu (d\mathfrak{s})=0 $ 
for all $ m\ge 1 $ by assumption. 
This equality also holds for $ m=0 $ 
because $ \PP ^{\mu }_{\mathfrak{s}}(\xi _{r}<\infty )=0 $ 
for $\mathfrak{s}\in \mathfrak{S}_{r}^{0}$. 
Hence by 
$ \mathfrak{S}= \sum_{m=0}^{\infty}\mathfrak{S}_{r,m}$ we deduce 
\begin{align}\label{:62f}&
\int _{\mathfrak{S}} 
\PP ^{\mu }_{\mathfrak{s}}(\xi _{r}<\infty )\mu (d\mathfrak{s}) = 0
.\end{align}
By \eqref{:62d} and \eqref{:62f} obtain  
$ \mu (\mathfrak{A}_{r})=0 $. 
\end{proof}

\begin{lem} \label{l:63}
Suppose $\mu (\mathfrak{A}_{r})=0 $. 
Then $ \mathrm{Cap}^{\mu }(\mathfrak{A}_{r})=0  $. 
\end{lem}
\begin{proof}
It is known that 
$ \mathrm{Cap}^{\mu }(\mathfrak{A}_{r})= 
\sup\{ \mathrm{Cap}^{\mu }(K); 
K\subset \mathfrak{A}_{r},\ K \text{ is compact }
\}  $ (see \cite[(2.1.6) in 66 p]{fot}). 
So let $ K $ be a compact set such that 
$ K \subset \mathfrak{A}_{r} $. 

Let $ \sigma _{K} =
\inf\{t>0;\mathfrak{X}_{t}\in K \}$ 
be the first hitting time to $ K $. 
Since $ K $ is compact, we deduce 
$ \mathfrak{X}_{\sigma _{K}}\in K $ if 
$ \sigma _{K}<\infty $. 

Suppose $ \mathfrak{s}\not\in \mathfrak{A}_{r} $. Then 
$ \PP ^{\mu }_{\mathfrak{s}}(\xi _{r}<\infty ) = 0 $ by \eqref{:62d}. 
Hence for $ \mathfrak{s}\not\in \mathfrak{A}_{r} $ 
\begin{align*}&
0 = 
\PP ^{\mu }_{\mathfrak{s}}(\xi _{r}<\infty ; 
\sigma _{K} < \xi _{r}<\infty )
= \int _{K}
\PP ^{\mu }_{\mathfrak{s}}(\mathfrak{X}_{\sigma _{K} }\in d\mathfrak{s}';
\sigma _{K}<\infty )
\PP ^{\mu }_{\mathfrak{s}'}(\xi _{r}<\infty )
.\end{align*}
This combined with \eqref{:62d} and 
$ K \subset \mathfrak{A}_{r}$ yields 
\begin{align}\label{:63a}&
\PP ^{\mu }_{\mathfrak{s}}(\mathfrak{X}_{\sigma _{K} }\in K ;
\sigma _{K}<\infty)=0 \text{ for }
 \mathfrak{s}\not\in \mathfrak{A}_{r} .
\end{align} 
Since $ \PP ^{\mu }_{\mathfrak{s}}(\mathfrak{X}_{\sigma _{K} }\in K ;
\sigma _{K}<\infty)= \PP ^{\mu }_{\mathfrak{s}}(\sigma _{K}<\infty )$, 
we deduce from \eqref{:63a} that 
\begin{align}\label{:63b}&
\PP ^{\mu }_{\mathfrak{s}}(\sigma _{K}<\infty )=0 \quad \text{ for }
\mathfrak{s}\not\in \mathfrak{A}_{r}
.\end{align}
By \eqref{:63b} and $ \mu (\mathfrak{A}_{r})=0 $ 
we have 
$ \int_{\mathfrak{S}}\PP ^{\mu }_{\mathfrak{s}}(\sigma _{K}<\infty )d\mu 
=0 $. From this we deduce 
$ \mathrm{Cap}^{\mu }({K})=0 $. 
We therefore obtain 
$ \mathrm{Cap}^{\mu }(\mathfrak{A}_{r})=0 $. 
\end{proof}

\noindent 
{\em Proof of \tref{l:25}. }
By \lref{l:61} we see that $ \PP ^{\nu }$ is conservative. 
Hence 
$ \int _{{S}\ts \mathfrak{S}} 
\PP ^{\nu }_{(x,\mathfrak{s})}(\eta < \infty ) d\nur = 0 $. 
Then by \lref{l:62} and \lref{l:63} we obtain 
$ \mathrm{Cap}^{\mu }(\mathfrak{A}_{r})=0 $ for all $ r\in \N $, which 
yields $ \mathrm{Cap}^{\mu }(\mathfrak{A}_{\infty })=0 $. 
Here 
\begin{align*}&
\mathfrak{A}_{\infty }=
\{ \mathfrak{s}\, ;\, \PP ^{\mu }_{\mathfrak{s}}(\xi _{\infty }<\infty )>0 \}
.\end{align*}
By $ \mathrm{Cap}^{\mu }(\mathfrak{A}_{\infty })=0 $ 
together with \eqref{:62c} and \eqref{:62a} 
we deduce \eqref{:24c}. 
\qed 

\section{Quasi-regularity: 
Proof of \lref{l:23}} \label{s:7}

In this section we prove the quasi-regularity of 
$ k $-labeled Dirichlet forms. 
So we begin by recalling the definition of 
quasi-regular by following \cite{mr}. 

Let $ E $ be a Polish space. 
A Dirichlet form $ (\E , \mathcal{D}) $ on 
$ L^2(E ,m) $ is called 
quasi-regular if it satisfies the following: 
\\
\thetag{Q.1} There exists an increasing sequence of compact sets 
$ \{ K_n \}  $ such that $ \cup _{n}\mathcal{D}(K_n) $ 
is dense in $ \mathcal{D}$ w.r.t.\ $ \E _{1}^{1/2} $-norm. 
Here $ \mathcal{D}(K_n) $ is the set of the elements $ f $ of $ \mathcal{D}$ 
such that $ f (x)=0 $ a.e.\ $ x\in K_n^c $, and 
$ \E _{1}^{1/2}(f)=\E (f,f)^{1/2}+ \|f\|_{L^2(E ,m)} $. 
\\
\thetag{Q.2} There exists a $ \E _{1}^{1/2} $-dense subset 
of $ \mathcal{D}$ whose elements have $ \E $-quasi continuous $ m $-version. 
\\
\thetag{Q.3} 
There exist a countable set $ \{ u_n \}_{n\in \N }  $ 
having $ \E $-quasi continuous $ m $-version $ \tilde{u}_n $, 
and an exceptional set $ \mathcal{N} $ such that 
$ \{ \tilde{u}_n \}_{n\in N}$ separates the points of 
$ E \backslash \mathcal{N}$.

Let $ (\Em , \dm ) $ be the closure of 
$(\Em , \dmi ) $ as before. 
By \thetag{M.2}, $ (\Em , \dm ) $ satisfies 
the quasi-regularity as seen in \lref{l:20}.  
We remark that $(\Em , \dm ) $ enjoys more strict conditions than 
the quasi-regularity. Indeed, we quote: 
\begin{lem}[\cite{o.dfa}] \label{l:71}
Assume \thetag{M.2}. Then we have the following.\\
\thetag{1} 
There exists a compact subset $ \{K_n\}_{n\in\N } $ 
such that $ \cup_{n}\di (K_n)$ is 
$ \{\Em _{1}\}^{1/2} $-dense in $ \dmi $. 
Here $ \di (K_n)=\{ f \in \di \, ;\, f (\mathfrak{s})= 0 
\text{ for all }\mathfrak{s}\in K_n^c \}  $ \\
\thetag{2} 
There exists countable elements $ \{ u_n \}_{n\in \N}$ 
of $ \dmi $ that separate the points of $ \mathfrak{S}$. \\
\thetag{3} If $ \3 $ is closable $ \Lm $, 
then the closure $ \4 $ is quasi-regular. 
\end{lem}

\begin{rem}\label{r:71} \thetag{1} 
We remark $ \di (K_n)\subset \dmi $. 
Indeed, $ \DDD [f ,g ] $ 
with $ f ,g \in \di (K_n)$ is bounded because 
$ \DDD [f ,g ] $ is continuous and $ K_n $ is compact. 
\\
\thetag{2} Suppose $ \3 $ is closable $ \Lm $. 
Then we see that \thetag{1} implies \thetag{Q.1}. 
Since $ \di (K_n)\subset C(\mathfrak{S}) $, 
\thetag{1} is more strict than \thetag{Q.1}. 
Moreover, \thetag{Q.2} is trivially satisfied 
in the above case because $ \dm $ is the closure of 
$ \dmi $ and $ \dmi \subset C(\mathfrak{S}) $. 
We see that \thetag{2} implies \thetag{Q.3} because 
$ \dmi \subset \dm $. 
The condition \thetag{2} is also more strict than 
\thetag{Q.3} in the sense that all $ \{u_n \}$ 
are continuous and $ \{u_n \}$ separate 
all the points of $ \mathfrak{S}$. 
\end{rem}

\begin{lem} \label{l:72}
There exists 
an increasing sequence of compact sets $ \{ K_{r,n} \}$ 
such that $ \cup_{n=1}^{\infty} \di (K_{r,n})$ 
is dense in $ \di ^{\murk } $ with respect to the 
$ \{\E ^{\murk } _{1}\}^{1/2} $-norm. 
\end{lem}
\begin{proof}
By \thetag{M.2} $ \murk $ becomes a finite measure. 
So the associated Dirichlet space is same as 
the Dirichlet space with the probability measure 
$ (\murk (\mathfrak{S}))^{-1}\murk $. 
We note here the measures in the energy form 
$ \E ^{\murk } $ and in the $ L^2 $-space are common. 

Applying \lref{l:71} to the measure 
$ (\murk (\mathfrak{S}))^{-1}\murk $ yields \lref{l:72}. 
\end{proof}

Recall that 
$ \dik = C_0^{\infty}({S}^{k})\ot\di $ and 
$ \dnuik = \{ f \in \dik \, ;\, 
\Enuk _1(f,f)< \infty \} $. 
\begin{lem} \label{l:73}
$ \cup_{r,n=1}^{\infty} 
C^{\infty}_{0}({S}^{k})\ot \di (K_{r,n})$ 
is dense in $ \dnuik  $ with respect to 
the $ \{\E ^{\nuk } _{1}\}^{1/2} $-norm. 
\end{lem}

\begin{proof}
Let $ \varphi \in C^{\infty}_{0}({S}^{k}) $ and 
$ f \in \di $ such that 
$ \varphi \ot f \in \dnuik $.  
It is sufficient for \lref{l:73} to show that 
for such an $ f $ and all $ \varepsilon >0 $ there exists 
$ f _{r,n} $ such that 
$ f _{r,n}\in \di (K_{r,n})$ 
and that 
\begin{align}\label{:73a}&
\E ^{\nuk } _{1} (\varphi \ot (f -f _{r,n}),
\varphi \ot (f -f _{r,n}))
\le \varepsilon
.\end{align}

Since $ \varphi \in C^{\infty}_{0}({S}^{k}) $, there exists 
an $ r $ such that $ \varphi = 0 $ on 
$ (\Srk )^{c} $. Hence there exists a constant 
$ \Ct{;73}=\cref{;73}(\varphi) $ such that 
\begin{align}\label{:73b}&
\nabla [\varphi ,\varphi ]( x ) \le 
\cref{;73}1_{\Srk } ( x ) 
, \  
\varphi ^2 ( x )\le 
\cref{;73}1_{\Srk } ( x )
\quad \text{ for all } x \in {S}^{k} 
.\end{align} 
We write 
$ \nabla [\varphi ]=\nabla [\varphi ,\varphi ] $ 
and 
$ \DDD [f ]=\DDD [f , f ] $. 
By a direct calculation we have 
\begin{align}\label{:73c}&
\quad \quad 
\E ^{\nuk } _{1} (\varphi \ot (f -f _{r,n})
,\varphi \ot (f -f _{r,n}))
\\ \notag 
&=\int_{{S}^{k} \ts \mathfrak{S}}  \{ 
\nabla [\varphi ]\ot |f -f _{r,n}|^2 +
\varphi ^2 \ot \DDD [f -f _{r,n}] + 
\varphi ^2\ot |f -f _{r,n}|^2 \} \, 
  d \nuk 
\\ \notag 
&\le 2\cref{;73}\int_{{S}^{k} \ts \mathfrak{S}}  1\ot \{
 |f -f _{r,n}|^2 +  \DDD [f -f _{r,n}]  
\}\,  d \nukr 
\\ \notag 
&= 2\cref{;73} \int_{\mathfrak{S}}  
\{ |f -f _{r,n}|^2 + 
\DDD [f -f _{r,n}] \} \,  d \murk 
\quad \text{ by }\eqref{:!1} 
\\ \notag 
&= 2\cref{;73} \E ^{\murk }_{1} 
(f -f _{r,n},f -f _{r,n})
.\end{align}
By \lref{l:72} we can take $ n $ and 
$ f _{r,n}\in \di (K_{r,n}) $ in such a way that 
$$ 
\E ^{\murk }_{1} (f -f _{r,n},f -f _{r,n})
\le \varepsilon /2\cref{;73} 
.$$ 
This combined with \eqref{:73c} yields 
\eqref{:73a}.
\end{proof}

\noindent 
{\em Proof of \lref{l:23}. }
\thetag{Q.1} follows from \lref{l:73}. 
\thetag{Q.2} is clear since $ \dnuk $ 
is the closure of $ \dnuik $ and 
$ \dnuik \subset C({S}^{k} \ts \mathfrak{S})$. 

For $ r \in \N $ 
let $ \{ u_{r,n} \}$ be a countable subset of 
$ \di ^{\murk } $ 
that separates the points of $ \mathfrak{S}$. 
We can obtain this by applying \lref{l:71} \thetag{2} to $ \murk $. 
We used here that \thetag{M.2} for $ \murk $ 
follows from that for $ \mu $. 

Let $ \{ \varphi _m\} $ be a countable subset of 
$ C_0^{\infty} ({S}^{k})$ 
that separates the points $ {S}^{k} $. 
Then by the same calculation as \eqref{:73c} we have 
\begin{align*}&
\E ^{\nuk } _{1} 
(\varphi _m \ot u_{r,n},\varphi _m \ot u_{r,n})
\le 
2\cref{;73} \E ^{\murk }_{1} 
(u_{r,n},u_{r,n})
.\end{align*}
Here $ \cref{;73} $ 
is a constant satisfying \eqref{:73b} 
for $ \varphi _m  $. Hence 
$ \varphi _m \ot u_{r,n} \in \dnuik $. 
Since $ \{ \varphi _m \ot u_{r,n} \}$ separates 
the points of $ {S}^{k}\ts \mathfrak{S}$, 
we obtain \thetag{Q.3}. 
\qed

\end{document}